\documentclass[article]{amsart}

\usepackage{Nyquist_sty}
\usepackage{xcolor}
\usepackage{algorithmic}
\usepackage{enumitem}
\usepackage{amsmath, bbm, amssymb, amsthm}

\numberwithin{equation}{section}

\begin{document}   

\title[LDP for MH chains]{A large deviation principle for the empirical measures of Metropolis-Hastings chains}


\author[F.~Milinanni]{Federica Milinanni}
\author[P.~Nyquist]{Pierre Nyquist}
\address[F.~Milinanni, P.\ Nyquist]{KTH Royal Institute of Technology}

\curraddr{Department of Mathematics, KTH, 100 44 Stockholm, Sweden}
\email{{fedmil@kth.se}, {pierren@kth.se}}
\thanks{}

\subjclass[2010]{60F10, 60C05; secondary  60G57, 60J05}

\keywords{Large deviations, empirical measure, Markov chain Monte Carlo, Metropolis-Hastings}

\begin{abstract} 
To sample from a given target distribution, Markov chain Monte Carlo (MCMC) sampling relies on constructing an ergodic Markov chain with the target distribution as its invariant measure. For any MCMC method, an important question is how to evaluate its efficiency. One approach is to consider the associated empirical measure and how fast it converges to the stationary distribution of the underlying Markov process. Recently, this question has been considered from the perspective of large deviation theory, for different types of MCMC methods, including, e.g., non-reversible Metropolis-Hastings on a finite state space, non-reversible Langevin samplers, the zig-zag sampler, and parallell tempering. This approach, based on large deviations, has proven successful in analysing existing methods and designing new, efficient ones. However, for the Metropolis-Hastings algorithm on more general state spaces, the workhorse of MCMC sampling, the same techniques have not been available for analysing performance, as the underlying Markov chain dynamics violate the conditions used to prove existing large deviation results for empirical measures of a Markov chain. This also extends to methods built on the same idea as Metropolis-Hastings, such as the Metropolis-Adjusted Langevin Method or ABC-MCMC. In this paper, we take the first steps towards such a large-deviations based analysis of Metropolis-Hastings-like methods, by proving a large deviation principle for the the empirical measures of Metropolis-Hastings chains. In addition, we also characterize the rate function and its properties in terms of the acceptance- and rejection-part of the Metropolis-Hastings dynamics.

\end{abstract}

\maketitle

\section{Introduction}

Sampling from a given probability distribution is an essential problem in a range of areas, for example biology, physics, epidemiology and ecology, and statistics. The most common approach is Markov chain Monte Carlo (MCMC), which allows the user to sample from a target probability distribution $\pi$, by generating an ergodic Markov chain $\{X_i\}_{i\geq 0}$ with $\pi$ as stationary distribution. These sampling techniques are particularly helpful when it is not possible to use methods that simulate directly from $\pi$, for example for computing posterior distributions in a Bayesian setting, or more generally when $\pi$ is only known up to a normalizing constant. Because of this, MCMC methods are now widely used across scientific disciplines, and are integral tools in areas such as computational chemistry and physics, statistics and machine learning \cite{RobertChristianP2004MCSM, AG07, AndrieuEtAl03}.

Because of their prevalence in a range of fields, the performance of MCMC algorithms has become an important topic within applied probability and computational statistics. In principle, even the standard Metropolis-Hastings algorithm \cite{Metropolis53, Hastings70} can be used to sample from essentially any target distribution $\pi$. However, when the underlying problem, and thus the distribution $\pi$, becomes more and more complex, convergence speed or the cost per iteration becomes an issue.  Analysing and improving the convergence speed of a given class of algorithms, as well as comparing the performance of different types of algorithms, is therefore not only interesting from a theoretical perspective, it is also of central importance for applications, where fast and accurate methods are needed for increasingly complex problems. 

When analysing performance of MCMC methods, the rate of convergence of time averages is a central quantity for comparing different metods, and for choosing hyperparameters. The fundamental idea underlying MCMC is that for an observable $f \in L^1(\pi)$, for an ergodic Markov chain $\{ X _i \} _{i \in \bN}$ with invariant distribution $\pi$, the $n$-step average $\frac{1}{n} \sum _{i=0} ^{n-1} f(X_i)$ can be used to approximate the expectation $\bE _\pi [ f(X)]$. This average can be viewed as the integral of $f$ with respect to the empirical measure of the Markov process. The rate of convergence of the empirical measure is therefore directly linked to the performance of a given MCMC method.

Because of the role the empirical measure plays in MCMC, and for Monte Carlo methods in general, in the past decade there has been an increasing interest in using the theory of large deviations for empirical measures to study the performance of MCMC methods \cite{DupuisDollLiu, Plattner11, Rey-Bellet2015,Rey-BelletLuc2015Vrfi,Rey-BelletLuc2016ItCo,DollDupuisNyquist,BierkensNyquistSchlottke, BierkensJoris2016NM}. However, surprisingly, existing large deviation results do not cover the empirical measure arising from the Metropolis-Hastings algorithm \cite{Metropolis53, Hastings70} on a general state space. Thus, in order to use a large deviation approach to analyse this foundational algorithm, or more advanced MCMC methods built on the same ideas as Metropolis-Hastings—such as the Metropolis-Adjusted Langevin Method (MALA) \cite{Besag94, RT96b, RR98} and methods based on Approximate Bayesian Computation (ABC) (see \cite{MMPS03, Beau19} for an overview and further references)—the relevant large deviation results must first be established. This is the main contribution of this paper: we prove the large deviation principle for the empirical measures associated with Markov chains arising from the Metropolis-Hastings algorithm. This sets the stage for future work proving similar results for Markov chains with dynamics that resemble those of Metropolis-Hastings, and for analysing the corresponding MCMC methods.  

The theory of large deviations has become a cornerstone in modern probability theory, with a wide range of applications. In the context of Monte Carlo methods, it has been known for a long time that for rare-event simulation, sample-path large deviations results are integral to analysing and designing efficient algorithms; see \cite{Bucklew04, AG07, BudhirajaAmarjit2019AaAo} and the references therein. In the MCMC setting, the theory remains much less explored for analysing performance and designing new, efficient methods. Instead, standard tools for convergence analysis of sampling methods based on ergodic Markov processes include: the spectral gap of the associated dynamics, mixing times of the process, asymptotic variance and functional inequalities (Poincar\'e, log-Sobolev) \cite{BedardRosenthal08, Rosenthal03, DiaconisHolmesNeal00, FHP+10, FdSH+93, HHMS+05}. However, these tools mainly provide information about convergence of the associated $n$-step transition operator or the law of the process, neither of which are directly linked to the convergence of the empirical measure. Empirical measure large deviations are instead concerned precisely with the convergence of the empirical measure. This is in turn linked to the transient behaviour of the underlying Markov chain, which is of central importance for the performance of MCMC methods. 

To the best of our knowledge, the first works on using large deviation theory to study the convergence of the empirical measures arising from MCMC sampling are \cite{DupuisDollLiu, Plattner11}. Therein, the authors analyse the performance of parallel tempering, one of the most frequently applied MCMC methods in computational chemistry and physics, from the perspective of large deviations, leading to the construction of a new type of method known as infinite swapping. In the subsequent work \cite{DollDupuisNyquist}, empirical measure large deviations and associated stochastic control problems are used to analyse the convergence properties of parallel tempering and infinite swapping. In \cite{DupuisWu} the authors study methods like parallel tempering and infinite swapping in the low-temperature regime, and use empirical measure large deviations to solve the long-standing open problem of optimal temperature selection. Similarly, in \cite{BierkensJoris2016NM, Rey-Bellet2015,Rey-BelletLuc2015Vrfi,Rey-BelletLuc2016ItCo} a large deviation approach is used to analyse certain irreversible samplers. In \cite{BierkensNyquistSchlottke}, large deviations for the empirical measures of certain piecewise deterministic Markov processes, including the zig-zag sampler, are obtained, and the associated rate function is used to address a key question concerning the optimal choice of the so-called switching rate of the zig-zag process. The results therein also highlight the differences in considering convergence of empirical averages, and in studying the convergence to equilibrium with, e.g., the spectral gap; see also \cite{Rosenthal03, VM20}.

In this paper we focus on the Metropolis-Hastings algorithm \cite{Metropolis53} (described in Section~\ref{sec:MH}), the most classical MCMC method and the main building block for many more advanced methods \cite{RobertChristianP2004MCSM,AG07, AndrieuEtAl03, Tierney98}. Because of its importance in the area of Monte Carlo sampling, the method is well-studied and classical results on convergence properties and performance include \cite{MengersenTweedie, RT96, GGR97, RR97, RR01, CRR05}; see also \cite{MT12, DMPS18} and the references therein for the general theory of Markov chains. However, despite significant efforts over long time, there are still gaps in our understanding of the theoretical properties of this fundamental class of algorithms. As an example, in a recent tour de force \cite{AndrieuChristophe2022Ecbf, AndrieuChristophe2022PifM} the authors develop a functional analytical framework, aimed at analysing Markov chains arising in sampling algorithms, and obtain the first explicit convergence bounds for the Metropolis algorithm. In \cite{BierkensJoris2016NM} a non-reversible version of Metropolis-Hastings is introduced and studied. One of the methods used for analysing performance is large deviations for the associated empirical measure. Because the setting is a finite state space $S$, the classical results \cite{DV75a, DV75b, DV76}, due to Donsker and Varadhan, give the large deviation principle. To the best of our knowledge, this is the only work that studies large deviations for Markov chains arising from algorithms of Metropolis-Hastings-type. In \cite{BierkensJoris2016NM} the focus is on the effects of non-reversibility, and there is thus no attempt of extending the large deviation results to the setting where the state space $S$ is instead a (uncountable) subset of $\bR ^d$. This is the setting typically encountered in applications.

The pioneering work by Donsker and Varadhan \cite{DV75a, DV75b, DV76} is often the starting point for empirical measure large deviations for Markov processes, and their results have been extended in numerous directions; see \cite{DZ94, FK06, BudhirajaAmarjit2019AaAo} and the references therein. However, it is pointed out in \cite{DupuisLiu2015} (see also Section \ref{sec:LDPEmpMeas}) that even for fairly simple continuous-time pure-jump processes, the results by Donsker and Varadhan, or more general versions of them such as in, e.g., \cite{KM05}, do not hold. This is because all such large deviation results rely on the transition probability function of the Markov process to have a density with respect to some reference measure. In \cite{DupuisLiu2015} the authors show how this condition can be replaced by a more general transitivity condition (Condition \ref{cond:DE}) to ensure that a large class of processes are covered. However, for the Metropolis-Hastings chains, neither of these conditions hold due to the rejection part of the dynamics. The purpose of this paper is to show that, despite this violation of the standard transitivity conditions, the empirical measures of the Metropolis-Hastings chain do satisfy a large deviation principle. The proof is based on the weak convergence approach \cite{DupuisEllis, BudhirajaAmarjit2019AaAo}, which is described in some more detail in Sections \ref{sec:LDPEmpMeas} and \ref{sec:LDPMH}. With the large deviation results established, our future work is aimed at (i) analysing the performance and comparing various Metropolis-Hastings algorithms using the rate function, and comparing the conclusion to, e.g., the recent results \cite{AndrieuChristophe2022Ecbf}; (ii) investigate whether optimal scaling results, similar to the celebrated results in \cite{GGR97, RR01}, can be obtained from a large deviation perspective; (iii) extend the results to cover more advanced MCMC algorithms, such as MALA and ABC-MCMC. These topics are all significant undertakings in their own right and we leave them to be investigated separately in future work.   

The remainder of the paper is organized as follows. In Section~\ref{sec:preliminaries} we provide the preliminaries needed for the paper: notation and definitions, a brief overview of large deviations for empirical measures, and a description of the Metropolis-Hastings algorithm. Next, in Section~\ref{sec:ass} we present the assumptions used for the Metropolis-Hastings chain. The main result is stated in Theorem~\ref{thm:LDP} in Section~\ref{sec:LDPMH}. In this section we also show some properties of the associated rate function. The proof of Theorem \ref{thm:LDP} is divided into two parts, in Sections~\ref{sec:upper} and \ref{sec:lower} we prove the Laplace upper and lower bound, respectively, which combined prove Theorem~\ref{thm:LDP}.

\section{Preliminaries}
\label{sec:preliminaries}
\subsection{Notation and definitions}
\label{sec:notation}
Throughout the paper we work with some probability space $\left( \Omega, \calF, \bP \right)$. We use a.s.\ and w.p.\ 1 as shorthand for almost sure, or almost surely, and with probability 1, respectively. 

For a Polish space $S$, with a translation invariant metric $d_S$, $\mathcal{B}(S)$ is the Borel $\sigma-$algebra on $S$, and 
$C(S)$ and $C_b(S)$ denote the spaces of functions $f: S \to \bR$ that are continuous, and bounded and continuous, respectively. For any $r \in \bR _+$ and $x \in S$, $B_r (x)$ is the open ball of radius $r$ with center in $x$:
\[
    B_r (x) = \{ y \in S: d_S(x,y) < r \}. 
\]
When $S \subseteq \bR ^d$, for some $d \geq 1$, we take $\lambda$ to denote Lebesgue measure on $\bR ^d$. We abuse notation a bit in that $\lambda$ is generically taken to represent Lebesgue measure, regardless of the underlying dimension $d$. For integration with respect to $\lambda$ we use the standard notation $dx$ for $\lambda(dx)$.

For a measure $\eta$ on $S$, and measurable function $f$ on $S$, we denote the integral of $f$ with respect to $\eta$ by $\eta(f) = \int _S f(x) \eta (dx) $. When $f$ is the indicator of a set $A$, we write $\eta (A) = \int _A \eta (dx)$.

The space of probability measures on $S$ is denoted by $\mathcal{P}(S)$. Given $\gamma\in\mathcal{P}(S^2)$, denote by $[\gamma]_1$ and $[\gamma]_2$ the first and second marginals of $\gamma$, respectively. For $\mu\in\mathcal{P}(S)$, define
\begin{equation}
    \label{eq:Amu}
    A(\mu)=\{\gamma\in\mathcal{P}(S^2):[\gamma]_1=[\gamma]_2=\mu\}.
\end{equation}
We consider the topology of weak convergence on $\calP (S)$: $\nu _n \to \nu$ in this topology if, for all $f \in C_b (S)$,
\[
  \nu _n (f) =   \int _S f(x) \nu _n (dx) \to \int _S f(x) \nu (dx) = \nu (f), \ n \to \infty.
\]
We use $\nu_n\Rightarrow\nu$ as shorthand notation for $\{\nu_n\}\subset \mathcal{P}(S)$ converging weakly to $\nu \in \mathcal{P}(S)$. Unless otherwise stated, we equip $\calP (S)$ with the 
\textit{L\'evy-Prohorov metric}, denoted $d_{LP}$: for $\nu, \mu \in \calP (S)$,
\[
d_{LP} (\nu, \mu) = \inf \left\{ \epsilon >0: \ \nu (A) \leq \mu(A ^\epsilon) + \epsilon, \ \textrm{for all closed subsets } A \subset S     \right\},
\]
where $A^\epsilon = \{ x \in S: d_S(x,A) < \e \}$. This metric is compatible with the topology of weak convergence (see, e.g., \cite{BudhirajaAmarjit2019AaAo}, Theorem A.1), and turns $\calP (S)$ into a Polish space. For any signed measure $\eta$ on $S$, the \textit{total variation norm} of $\eta$, $\norm{\eta} _{TV}$, is defined as
\[
	\norm{\eta} _{TV} = \sup _{f} \left| \eta (f) \right|, 
\]
where the supremum is taken over all measurable functions bounded by 1. For $\nu, \mu \in \calP(S)$, the total variation norm provides an upper bound on $d_{LP}$:
\[
	d_{LP} (\nu, \mu) \leq \norm{\nu - \mu} _{TV}.
\]

For a measurable space $(\calY, \calA)$, let $q(y,dx)$ be a collection of probability measures on $S$ parameterized by $y \in Y$. Then $q$ is called a \textit{stochastic kernel} on $S$ given $\calY$ if, for every $A \in \calB (S)$, the map $y \mapsto q(y,A) \in [0,1]$ is measurable.

For a Markov chain $\{ X_i \} _{i \in\bN}$ taking values in $S$, for a given $x_0 \in S$, we denote by $\bP _{x_0}$ the distribution of $\{ X_ i \} _{i \in\bN}$  starting at $x_0$. The associated expectation operator is denoted by $\bE _{x_0}$. The \textit{transition probability function}, or \textit{transition kernel}, of a Markov chain is a stochastic kernel $q$, such that the distribution of $X_i$ given $X_{i-1}$ is given by $q(X_{i-1}, \cdot)$. We say that a transition probability function $q(x,dy)$ on $S \times \calP (A)$ satisfies the \textit{Feller property} if, for any sequence $\{ x_n \} _{n \in \bN}$ such that $x_n \to x \in S$ as $n\to \infty$, $q(x_n, \cdot) \Rightarrow q(x, \cdot)$.

Given a measure $\mu\in\mathcal{P}(S)$ and a transition kernel $q(x,dy)$, we say that $\mu$ is \textit{invariant} for $q$, or for the corresponding Markov chain, if for all $A\in\mathcal{B}(S)$,
\begin{equation*}
    \mu(A)=\int_Sq(x,A)\mu(dx).
\end{equation*}

For $\nu \in \calP (S)$, $R(\cdot \parallel \nu) : \calP (S) \to [0, \infty]$ is the \textit{relative entropy} (with respect to $\nu$), defined by
\begin{equation*}
    R(\mu\parallel \nu) = \begin{cases} \int_S\log \left( \frac{d\mu}{d\nu}\right) d\mu, &  \mu \ll \nu, \\
    + \infty, & \textrm{otherwise.}
    \end{cases}
\end{equation*}
We recall the following properties of relative entropy (see Lemmas 1.4.1 and 1.4.3 in \cite{DupuisEllis}): $R(\cdot\parallel\cdot)$ is jointly convex and jointly lower semi-continuous with respect to the weak topology on $\calP (S) ^2$, and $R(\mu \parallel \nu) = 0$ if and only if $\mu = \nu$. Another  useful property follows from the chain rule for relative entropy (see Theorem~2.6 and Corollary~2.7 in \cite{BudhirajaAmarjit2019AaAo}): given two transition kernels $p,q$, for any $\mu \in \calP (S)$, 
\begin{equation}
\label{eq:chainRuleCorollary}
    R(\mu\otimes p\parallel \mu\otimes q)=\int_SR(p(x,\cdot)\parallel q(x,\cdot))\mu(dx).
\end{equation}

Lastly, for a set $A$, $A^\circ$ and $\bar A$ denote the \textit{interior} and \textit{closure} of the set, respectively, and $x \mapsto I\{ x \in A \}$ is the indicator function of the set $A$. When the set is a singleton, $A=\{y\}$, we write $I\{x=y\}$. We also use $\delta _y$ to denote this case.

\subsection{Large deviations for empirical measures of a Markov chain}
\label{sec:LDPEmpMeas}
Consider a Markov chain $X = \{X_i \} _{i \geq 0}$ with state space $S$ and transition probability function $p$. The \textit{empirical measure}, $L ^n$, associated with the chain $X$ is defined as 
\begin{align}
 \label{eq:empiricalMeasure}
    L^n (\cdot) = \frac{1}{n} \sum _{i=0} ^{n-1} \delta _{X_i} (\cdot).  
\end{align}
For each  $n$, this is a random element of $\calP(S)$. We can also view $\{ L^n \} _{n \geq 0}$ as a stochastic process in $\calP (S)$. 

In the context of MCMC methods, empirical measures are essential objects as they are used for forming approximations for any observable: for a given observable $f \in C_b (S)$, we have 
\[
    L^n (f) = \frac{1}{n} \sum _{i=1} ^n f(X_i).
\]
If the Markov chain $X$ has an invariant distribution $\pi \in \calP (S)$ and is ergodic, we have $L^n (f) \to \pi (f)$, a.s.\ as $n \to \infty$. Thus, there is a direct link between the convergence properties of the empirical measure $L^n$ and the performance of Monte Carlo methods based on time averages for approximating observables. 

Classical methods for studying performance of MCMC methods are often mixing properties or asymptotic variance, which are not directly linked to the empirical measure $L^n$ of the underlying Markov chain. The theory of large deviations on the other hand, is concerned precisely with deviations of $L^n$ from $\pi$ as the number of steps $n$ grows. It therefore serves as a useful complement to the more traditional methods for analysing performance of a given MCMC method, as well as for designing new algorithms.

At the heart of the theory of large deviations is the \textit{large deviation principle} (LDP): the sequence $\{ L^n \}$ is said to satisfy an LDP with speed $n$ and \textit{rate function} $I : S \to [0, \infty]$, if $I$ is lower semi-continuous, has compact sub-level sets and for any measurable $A \subset \calP (S)$,
\begin{align*}
    - \inf _{\nu \in A ^\circ} I(\nu) &\leq \liminf _{n \to \infty} \frac{1}{n} \log \bP (L ^n \in A ^\circ) \\
    &\leq \limsup _{n \to \infty} \frac{1}{n} \log \bP (L ^n \in \bar A ) \leq - \inf _{\nu \in \bar A} I (\nu).    
\end{align*}
The gist of these inequalities is that, if $\{ L^n \}$ satisfies an LDP with speed $n$ and rate function $I$, then for any $\nu \in \calP (S)$ and $n$ large, \[
    \bP (L ^n \approx \nu) \simeq \exp\{-n I(\nu)\}.
\]
The definition of an LDP makes this statement rigorous in the limit $n \to \infty$. 

For any metric space, an equivalent formulation of the LDP is the \textit{Laplace principle} (see e.g., Theorems 1.5 and 1.8 in \cite{BudhirajaAmarjit2019AaAo}). In the setting of the empirical measures $\{ L^n \}$, we have that this sequence satisfies a Laplace principle, with speed $n$ and rate function $I$ (same as in the LDP), if for any $F \in C_b \left(\calP (S) \right)$,
\begin{align}
\label{eq:LaplcePrinc}
    \lim _{n \to \infty} \frac{1}{n} \log \bE \left[ e^{-n F(L^n)} \right] = - \inf _{\nu \in \calP (S)} \left\{ F(\nu) + I(\nu) \right\}.
\end{align}

The starting point for large deviations of empirical measures of Markov processes is the pioneering work of Donsker and Varadhan \cite{DV75a, DV76}. A central assumption in those works is that the transition probability function $p$ has a density with respect to some reference measure. This is a reasonable transitivity assumption for processes that involve something that, in some sense, resembles a diffusive term. However, in \cite{DupuisLiu2015} the authors show that it is a rather restrictive condition and as an example construct a simple continuous-time pure-jump process for which it does not hold. The following alternative condition on $p$ was used in \cite{DupuisLiu2015} to establish an LDP for the empirical measures of a Markov process.
\begin{condition}[Condition 6.3 in \cite{BudhirajaAmarjit2019AaAo}]
\label{cond:DE}
The transition kernel $p$ of the Markov chain $X$ is such that there exist positive integers $l_0$ and $n_0$, such that for all $x$ and $\zeta$ in $S$,
\begin{equation}
    \label{eq:transCond}
    \sum_{i\ge l_0}2^{-i}p^{(i)}(x,dy)\ll \sum_{j\ge n_0}2^{-j}p^{(j)}(\zeta,dy),
\end{equation}
where $p^{(k)}$ denotes the $k-$step transition probability.
\end{condition}
This condition is general enough to cover a large class of Markov processes, both in discrete and continuous time; see e.g., \cite{BudhirajaAmarjit2019AaAo, DupuisLiu2015} and the references therein. However, it does not cover the case when $X$ comes from a Metropolis-Hastings scheme, as we show with a simple counterexample in Section \ref{sec:LDPMH}. Condition \ref{cond:DE}, or variations of it, is a key ingredient in existing work on large deviations for Markov chains. Because it is not satisfied for Metropolis-Hastings, in order to use large deviations to analyse the performance of such algorithms, and with an outlook towards more advanced MCMC methods that build on the Metropolis-Hastings algorithm—e.g.,\  MALA and ABC-MCMC—we must first establish the relevant LDP. This is the main contribution of this paper.

\subsection{Metropolis-Hastings algorithm}
\label{sec:MH}
We now give a brief description of the Metropolis-Hastings (MH) algorithm for constructing a Markov chain $X = \{ X_i \}_{i \geq 0}$ with the target measure $\pi$ as invariant distribution. For simplicity we restrict ourselves to the setting where $S \subseteq \bR ^d$ and $\pi$ is equivalent to Lebesgue measure. More abstract settings are possible as well, see for example \cite{Tierney98}. However this would require different assumptions and modifications of the proof of the large deviation principle in Section \ref{sec:LDPEmpMeas}. 

The main ingredient of the MH algorithm is the \textit{proposal distribution} $J(\cdot|x)\in\mathcal{P}(S)$, defined for all $x\in S$. If the chain after $n$ steps is in some state $X_n = x_n$, a proposal $Y_{n+1}$ for the next state $X_{n+1}$ is generated from $J(\cdot | x_n)$. This is followed by an acceptance-rejection step, which is defined in terms of the \textit{Hastings ratio}, 
\begin{equation*}
    \varpi(x,y)=\min\left\{1,\frac{\pi(y)J(x|y)}{\pi(x)J(y|x)}\right\};
\end{equation*}
where if $\pi (x) J(y|x) = 0$, we set $\varpi (x,y) = 1$. The proposed move from $X_n = x_n$ to $X_{n+1} = Y_{n+1}$ is accepted with probability $\varpi (x_n, Y_{n+1})$, and rejected with probability $1 - \varpi(x_n, Y_{n+1})$. In the latter case, we set $X_{n+1} = x_n$. The pseudocode for the update step in the MH algorithm is presented in Algorithm~\ref{alg:MH}.
 \begin{algorithm}
\caption{Metropolis-Hastings algorithm}
\label{alg:MH}
\begin{algorithmic}[1]
\REQUIRE Given $X_i=x_i$,
 \STATE Generate a proposal $Y_{i+1}\sim J(\cdot|x_i)$
 \STATE Set
    \begin{equation*}
        X_{i+1}=\begin{cases}
            Y_{i+1}\quad \text{with probability }\varpi(x_i,Y_{i+1})\\
            x_i\quad \text{with probability }1-\varpi(x_i,Y_{i+1})
        \end{cases}
    \end{equation*}
\end{algorithmic}
\end{algorithm}

Define a transition kernel $a(x, dy)$ and a function $r : S \to [0,1]$ by
\begin{equation}
    \label{eq:acceptancePart}
    a(x,dy) = \min\left\{1,\frac{\pi(y)J(x|y)}{\pi(x)J(y|x)}\right\}J(dy|x),
\end{equation}
and
\begin{equation}
\label{eq:rejectionPart}
    r(x) =1-a(x,S)=1-\int_S a(x,dy).
\end{equation}
The kernel $a$ corresponds to the acceptance-part of the MH algorithm, i.e., it corresponds to transitions to proposed states that are accepted in the MH algorithm. Similarly, $r$ corresponds to the rejection part: it represents the probability of rejecting a proposed state, and thus remaining at the current state of the chain. With these definitions, the dynamics of the MH algorithm corresponds to generating a Markov chain $\{ X_i \} _{i \geq 0}$, the MH chain, with transition kernel
\begin{equation}
    \label{eq:MHtransitionKernel}
    K(x,dy) = a(x,dy)+r(x)\delta_x(dy).
\end{equation}
For a more in-depth look at the MH algorithm and its various properties, see for example \cite{RobertChristianP2004MCSM} and the references therein. A key observation is that due to the form of the Hastings ratio, and the corresponding kernel $K$, under reasonable assumptions on the proposal distribution $J$, the MH chain $\{ X _i \} _{i \geq 0}$ generated according to the above has $\pi$ as its unique invariant measure.

\section{Assumptions}
\label{sec:ass}
In this section, we state the assumptions we make on the MH chain defined in Section \ref{sec:MH}. Rather than aiming to make them as general as possible, we have aimed for assumptions, primarily on the proposal distribution $J$, that are tangible from the perspective of MCMC methods. One alternative, commonly used when studying this type of Markov chain, is to assume the existence of some Lyapunov function \cite{MengersenTweedie, RT96b, KM03, KM05}. Although this ensures the convergence of the empirical measures, however for the large deviation results additional assumptions are still needed; see e.g., the Donsker-Varadhan-like assumption on the transition kernel in \cite{KM05}.    

As mentioned in Section~\ref{sec:MH}, we make the assumption that $S \subseteq \bR ^d$, for some $d\geq 1$. We make a slight abuse of notation, in that we let $\pi(\cdot)$, $J(\cdot|x)$, and $a(x,\cdot)$ denote both the corresponding measures and probability density functions. 
In order to establish the LDP, we make the following additional assumptions.

\begin{enumerate}[label = (A.\arabic*)]
\item{\label{ass:targetAbsContLambda}
    The target probability measure $\pi$ is equivalent to $\lambda$ on $S$ (i.e., $\pi\ll\lambda$ and $\lambda\ll\pi$). The probability density $\pi(x)$ is a continuous function.}
\item{\label{ass:proposalDistributionAbsCont}
    The proposal distribution $J(\cdot|x)$ is absolutely continuous with respect to the target measure $\pi$ (i.e., $J(\cdot|x)\ll\pi$), for all $x\in S$. The probability density  $J(y|x)$ is a continuous and bounded function of $x$ and $y$, and it satisfies
    \begin{equation}
    \label{eq:Jpositive}
        J(y|x)>0,\quad \forall (x,y)\in S^2.
    \end{equation}}

\item{\label{ass:compactSpaceOrLyapunov}
    There exists a Lyapunov function $U:S\to[0,\infty]$ such that the following properties hold:
    \begin{enumerate}[label=(\alph*)]
        \item $\inf_{x\in S}\left[U(x)-\log\int_Se^{U(y)}K(x,dy)\right]>-\infty$
        \item For each $M<\infty$, the set
        \begin{equation*}
            \left\{x\in S \,:\,U(x)-\log\int_Se^{U(y)}K(x,dy)\le M\right\}
        \end{equation*}
        is a relatively compact subset of $S$.
        \item For every compact set $K\subset S$ there exists $C_K<\infty$ such that 
        \begin{equation*}
            \sup_{x\in K}U(x)\le C_K.
        \end{equation*}
    \end{enumerate}}
\end{enumerate}

Because $\pi$ and $\lambda$ are equivalent measures, the support of $\pi$ is all of $S$. However, it is not necessarily the case that $\pi(x) > 0$ for all $x\in S$, as there may exist a (nonempty) set $E \subset S$, such that $\lambda (E) = 0$ and $\pi(x)=0$ for $x \in E$. Therefore, define the set $S_+$ as
\begin{equation}
    \label{eq:Spositive}
    S_+ = \{y\in S\,:\, \pi(y)>0\}.
\end{equation}
Observe that $S_+$ is an open subset of $S$, being the density function $\pi(x)$ continuous.

Assumptions \ref{ass:targetAbsContLambda}-Assumption~\ref{ass:proposalDistributionAbsCont} are used to show that the MH transition kernel $K$, and thus the MH chain $\{ X_i \}_{i \geq 0}$, has certain properties. Assumption \ref{ass:compactSpaceOrLyapunov} replaces a compactness-assumption on $S$ for proving the LDP. In the case of a compact state space $S$, this assumption is not needed.

\begin{remark}
\label{rmk:continuityAandR}
We start by showing that the combination of \ref{ass:targetAbsContLambda} and \ref{ass:proposalDistributionAbsCont} ensure continuity and boundedness of the components $a$ (acceptance part) and $r$ (rejection part) of the MH transition kernel $K$. To see this, note first that it is sufficient to define $K$ only for the states $x \in S_+$. If $x\in S_+$, the quantity 
\begin{equation}
    \min\left\{1,\frac{\pi(y)J(x|y)}{\pi(x)J(y|x)}\right\}
\end{equation}
is well defined, and therefore so is the MH transition kernel $K(x,\cdot)$. Moreover, if the initial point $X_0$ of the chain belongs to $S_+$, the MH algorithm only allows moves to states $y$ that preserve the property $\pi(y)>0$.

Consider now $x\in S_+$. Assumption~\ref{ass:targetAbsContLambda} and Assumption~\ref{ass:proposalDistributionAbsCont} imply that $J(\cdot|x)\ll\lambda$. Therefore, the acceptance part \eqref{eq:acceptancePart} of $K(x,\cdot)$ is absolutely continuous with respect to the Lebesgue measure (i.e., $a(x,\cdot)\ll\lambda$), and its density is given by
    \begin{equation}
        \label{eq:acceptanceDensity}
        a(x,y) = \min\left\{1,\frac{\pi(y)J(x|y)}{\pi(x)J(y|x)}\right\}J(y|x).
    \end{equation}   
    Since $\pi(x)$ is continuous for all $x\in S$ and $J(x|y)$ is continuous and bounded for all $(x,y)\in S^2$, we have $a(x,y) \in C_b(S_+ \times S)$. 

From the continuity of $a(x,y)$ on $S_+\times S$, we obtain that $r(x)=1-a(x,S)$ is also continuous for all $x\in S_+$. This continuity extends to all of $S$. First, if $x\notin S_+$, so that $\pi(x)=0$, then 
\begin{equation*}
    r(x)=1-a(x,S)=
    1-\int_SJ(y|x)dy=0,
\end{equation*}
since $\pi(y)>0$ for $\lambda$-almost all $y\in S$. Take $x\notin S_+$ and a sequence $\{x_n\}\subset S$ that converges to $x$. From the continuity of the target density function $\pi$, $\pi(x_n)\to\pi(x)=0$. Moreover, for a fixed $y$ such that $\pi(y)>0$, we have $a(x_n,y)\to J(y|x)$ as $n \to \infty$. To see this, note that since $\pi$ and $\lambda$ are equivalent, $\pi(y)>0$ for $\lambda-$almost all $y\in S$. It follows that $\lim_{n\to\infty}a(x_n,y)=J(y|x)$ for $\lambda-$almost all $y\in S$. Recalling that $J$ is bounded, by dominated convergence we have
\begin{equation*}
    \lim_{n\to\infty}a(x_n,S)= \lim_{n\to\infty}\int_Sa(x_n,y)dy=\int_S\lim_{n\to\infty}a(x_n,y)dy=\int_S J(y|x)dy=1.
\end{equation*}
This in turn implies that 
\[ 
\lim_{n\to\infty}r(x_n)=1-\lim_{n\to\infty}a(x_n,S)=0.
\] 
Since $r(x) = 0$, this shows that $r$ is continuous on $S$. 
\end{remark}

\begin{remark}
\label{rmk:KhasAlawaysAcceptancePart}
Next, we show that \ref{ass:targetAbsContLambda}-\ref{ass:proposalDistributionAbsCont} ensure that $K$ has the target measure $\pi$ as its unique invariant distribution, and the MH chain $\{ X_i \}_{i \geq 0}$ is ergodic.

Let $x\in S_+$ as defined in \eqref{eq:Spositive}.
    Since $\lambda\ll\pi$ by \ref{ass:targetAbsContLambda}, $\pi(y)>0$ for $\lambda-$almost every $y\in S$. Moreover, by Assumption~\ref{ass:proposalDistributionAbsCont}, $J(x|y)>0$ for all $(x,y)\in S^2$. It follows that $a(x,y)>0$ for $\lambda-$a.e. $y\in S$. This in turn implies that $\lambda\ll a(x,\cdot)$, and $\lambda$ and $a(x,\cdot)$ are equivalent measures for all $x\in S_+$. By transitivity, $a(x,\cdot)$ and $a(y,\cdot)$ are equivalent for all $x,y,\in S_+$. We now show that from this it follows that the MH transition kernel $K$ is indecomposable, i.e. there are no disjoint Borel sets $A_1,A_2\in\mathcal{B}(S)$ such that
    \begin{equation*}
        K(x,A_1)=1\;\forall x\in A_1\quad\text{and}\quad K(y,A_2)=1\;\forall y\in A_2.
    \end{equation*}
    We argue by contradiction. Assume that two such sets exist. Then, 
    \begin{equation}
    \label{eq:indecomposibilityArgument}
        1=K(x,A_1)=a(x,A_1)+r(x)\delta_x(A_1).
    \end{equation}
    Since $\lambda\ll a(x,\cdot)$, we have $a(x,S)>0$, and thus $r(x)=1-a(x,S)<1$ for all $x \in s$. Combined with \eqref{eq:indecomposibilityArgument}, this shows $a(x,A_1)>0$. It follows from $a(x,\cdot)$ and $a(y,\cdot)$ being equivalent measures that $a(y,A_1)>0$, which contradicts the assumption. Hence, $K(x,dy)$ is indecomposable. By Theorem 7.16 in~\cite{Breiman}, $\pi$ is the unique invariant distribution for the MH transition kernel $K(x,dy)$ and the Markov chain associated with $\pi$ and $K(x,dy)$ is ergodic.
    
\end{remark}

\begin{remark}
    For the case $S=\mathbb{R}^d$, Section~8.2 in \cite{DupuisEllis} describes a class of models for which a Lyapunov function $U$ that satisfies \ref{ass:compactSpaceOrLyapunov} exists. Here we present their example adapted to the MH kernel $K$. For specific choices of $J$ and/or $\pi$, this assumption can be made more explicit (or verified). 
    
    Let $b:\mathbb{R}^d\to\mathbb{R}^d$ be measurable. Denote by $\langle\cdot,\cdot\rangle$ the scalar product in $\mathbb{R}^d$ and for $\alpha\in\mathbb{R}^d$ define
    \begin{equation*}
        H_b(x,\alpha)=\log\left[\int_{\mathbb{R}^d}e^{\langle\alpha,y-x-b(x)\rangle}a(x,dy)+r(x)e^{-\langle\alpha,b(x)\rangle}\right].
    \end{equation*}
    Consider the following assumptions.
    \begin{enumerate}[label=(\alph*)]
        \item\label{Lyap1} $b$ is bounded on all compact sets in $\mathbb{R}^d$
        \item\label{Lyap2} there exists $r>0$ such that 
        \begin{equation*}
            \sup_{x\in\mathbb{R}^d}H_b(x,\alpha)<\infty,
        \end{equation*}
        for all $\alpha\in\mathbb{R}^d$ that satisfy $\|\alpha\|\le r$
        \item\label{Lyap3} there exists a Lipschitz continuous function $U:\mathbb{R}^d\to[0,\infty)$ for which
        \begin{equation*}
            \lim_{\|x\|\to\infty}\left[U(x+b(x))-U(x)\right]=-\infty.
        \end{equation*}
    \end{enumerate}
    If \ref{Lyap1}, \ref{Lyap2} and \ref{Lyap3} hold, then $U$ is a Lyapunov function as required by Assumption\ref{ass:compactSpaceOrLyapunov}.

    A natural choice for $b$ is 
    \begin{equation*}
        b(x)=\int_{\mathbb{R}^d}y\cdot a(x,dy)-(1-r(x))\cdot x,
    \end{equation*}
    and the corresponding $H_b$ is
    \begin{equation*}
        H_b(x,\alpha)=\log e^{-\langle\alpha,\int a(x,dy)+r(x)\cdot x\rangle}+\log\left[\int_{\mathbb{R}^d} e^{\langle\alpha,y\rangle}a(x,dy)+r(x)e^{\langle\alpha,x\rangle}\right].
    \end{equation*}
Note that if the space $S$ is compact, then Assumption~\ref{ass:compactSpaceOrLyapunov} is automatically satisfied (for example, take $U(x) \equiv 0$).
\end{remark}

\section{Large deviations for empirical measures of Metropolis-Hastings chains}
\label{sec:LDPMH}
We are now ready to state our main result, an LDP for the sequence $\{ L^n\}$ of empirical measures of the MH chain $\{ X_i \} _{i \geq 0}$ with invariant distribution $\pi$ (see Section \ref{sec:MH} for the definition).

\begin{theorem}
    \label{thm:LDP}
    Let $\{X_i\}_{i\ge 0}$ be the Metropolis-Hastings chain from Section \ref{sec:MH} and $K(x,dy)$ the associated transition kernel. Let $\{L^n\}_{n\ge 0}\subset\mathcal{P}(S)$ be the corresponding sequence of empirical measures, defined in \eqref{eq:empiricalMeasure}. Under Assumptions \ref{ass:targetAbsContLambda}-\ref{ass:compactSpaceOrLyapunov}, with $A(\mu)$ as in \eqref{eq:Amu}, $\{L^n\}_{n\ge 0}$ satisfies an LDP with speed $n$ and rate function
    \begin{equation}
        \label{eq:rateFunction}
        I(\mu)=\inf_{\gamma\in A(\mu)} R(\gamma \parallel \mu\otimes K).
    \end{equation}
\end{theorem}

As mentioned in Section \ref{sec:notation}, we consider $\calP (S)$ as a metric space (equipped with, e.g., the L\'evy-Prohorov metric). Therefore, the LDP is equivalent to the Laplace principle, and we will use the latter to prove Theorem \ref{thm:LDP}. More specifically, the proof is split up into proving the \textit{Laplace principle upper bound},
\begin{align}
\label{eq:upper}
    \liminf _{n \to \infty} - \frac{1}{n} \bE _{x_n} \left[ e^{-n F (L^n)} \right] \geq \inf _{\mu \in \calP (S)} \left\{ F(\mu) + I(\mu) \right\},
\end{align}
and the \textit{Laplace principle lower bound},
\begin{align}
\label{eq:lower}
    \limsup _{n \to \infty} - \frac{1}{n} \bE _{x} \left[ e^{-n F (L^n)} \right] \leq \inf _{\mu \in \calP (S)} \left\{ F(\mu) + I(\mu) \right\},
\end{align}
for every $F \in C_b (\calP (S))$, every sequence $\{x_n\}\subset S$ and $x\in S$. 
The respective proofs are given in Sections \ref{sec:upper} and \ref{sec:lower}. The starting point for both bounds is the following representation formula (Proposition 6.1 in \cite{BudhirajaAmarjit2019AaAo}): for every bounded, measurable $F : \calP (S) \to \bR$,
\begin{align}
\label{eq:rep}
    -\frac{1}{n} \log \bE \left[ e^{-nF(L^n)} \right] &= \inf _{\{ \bar \mu ^n _i \} } \bE \left[ F(\bar L ^n) + \frac{1}{n} \sum_{i=1} ^n R(\bar \mu ^n _i \parallel K(\bar X ^n _i, \cdot)) \right],
\end{align}
where $\bar L^n$ is the controlled empirical measure, $\bar L ^n = \frac{1}{n} \sum _{i=0} ^{n-1} \delta _{\bar X ^n}$, and the conditional distribution of $\bar X^n _i$ given $\sigma(\bar X^n_1,\dots,\bar X^n_{i-1})$ is $\bar \mu ^n _i$. The infimum is over all such controls, i.e., random probability measures, $\bar \mu ^n _i $, such that $\bar \mu ^n _i$ is measurable with respect to $\calF _{i-1} ^n = \sigma \left( \bar X ^n _1, \dots, \bar X^n _{i-1} \right)$, with $\calF ^n _0 = \{\emptyset, \Omega\}$; see \cite{DupuisEllis, BudhirajaAmarjit2019AaAo} for more details.

For the upper bound, under Assumptions \ref{ass:targetAbsContLambda}-\ref{ass:compactSpaceOrLyapunov}, the proof of Proposition 6.13 in \cite{BudhirajaAmarjit2019AaAo}, with the additional arguments in Section 6.10 therein to account for a non-compact state space, can be applied in our setting as well. The only thing that needs to be verified is the Feller property of the MH transition kernel $K$ (see Lemma \ref{lem:MHKernelFeller}).  

The work for proving Theorem \ref{thm:LDP} lies in proving the lower bound \eqref{eq:lower}. Existing results rely on some variation of Condition \ref{cond:DE}. However, such a condition is not applicable in our setting, as the following simple example shows: Take an $x \in S$ such that $r(x) > 0$ (i.e., when in $x$, there is a positive probability of rejecting a proposed move and stay in $x$) and consider the Borel set $A=\{x\}\in \mathcal{B}(S)$. If $x\neq \zeta$, then $p^{(j)}(\zeta,A)=p^{(j)}(\zeta,x) = 0,\forall j\ge 0$. However, $p^{(i)}(x,x)>0, \forall i\ge 0$, since $r(x)>0$ . This shows that \eqref{eq:transCond} does not hold for all $x\in S$, and Condition \ref{cond:DE} does not hold for the MH kernel $K$, nor for kernels of similar type, such as those arising in ABC-MCMC or MALA.

In Section \ref{sec:lower} we show how the Laplace principle lower bound can be shown for the MH chain without relying on a transitivity assumption like Condition \ref{cond:DE}. The main point is that due to the specific structure of the MH kernel, under Assumptions \ref{ass:targetAbsContLambda}-\ref{ass:proposalDistributionAbsCont} the chain retains the properties that are important for proving the LDP (and typically guaranteed by something like Condition \ref{cond:DE} combined with other assumptions). 

The main difficult in the proof arises from the fact that, contrary to the setting in \cite{BudhirajaAmarjit2019AaAo}, for $\nu \in \calP (S)$, $I(\nu) < \infty  $ does not imply that $\nu \ll \pi$. In \cite{BudhirajaAmarjit2019AaAo} this implication is used in defining near-optimal controls in the representation \eqref{eq:rep}, which in turn can be used to prove the lower bound. 

Before proceeding with the proofs of the upper and lower bounds, in the following section we give some different characterizations and properties of the rate function $I$ in \eqref{eq:rateFunction}.

\subsection{Characterization and properties of the rate function}
We first express the rate function \eqref{eq:rateFunction} in a more convenient form. By Lemma~6.8(a) in~\cite{BudhirajaAmarjit2019AaAo}, the probability measures in the set $A(\nu)$ are of the form 
\begin{equation*}
    \gamma(dx\,\times\,dy)=\nu(dx)\,q(x,dy),
\end{equation*} for a transition kernel $q(x,dy)$ such that $\nu$ is invariant for $q$. Therefore, using \eqref{eq:chainRuleCorollary}, the chain rule for relative entropy, we can rewrite \eqref{eq:rateFunction} as
\begin{equation}
\label{eq:rateFuncAsInf}
I(\nu)=\inf_{q\in\mathcal{Q}}\int_SR(q(x,\cdot)\parallel K(x,\cdot))\nu(dx),
\end{equation}
where $\mathcal{Q}$ denotes the set of all the transition kernels $q(x,dy)$ on $S$ such that $\nu$ is an invariant distribution for $q$. Lemma~6.8(b) in \cite{BudhirajaAmarjit2019AaAo} guarantees the existence of a minimizing $q$ in the definition of $I(\nu)$, under the assumption $I(\nu)<\infty$. That is, there exists a transition kernel $q$ with stationary distribution $\nu$ such that
\begin{equation}
    \label{eq:rateFunctionIntegral}
    I(\nu)=\int_SR(q(x,\cdot)\parallel K(x,\cdot))\nu(dx).
\end{equation}

The representation \eqref{eq:rateFunctionIntegral} of the rate function allows us to characterize the minimizers $q$, based on the form of the MH transition kernel $K$ \eqref{eq:MHtransitionKernel}, as the following result shows.

\begin{lemma}
    \label{lem:transitionKernelForm}
    If $I(\nu)<\infty$, then the transition kernel $q(x,dy)$ in \eqref{eq:rateFunctionIntegral} is $q(x,\cdot)\ll K(x,\cdot)$ $\nu$-a.s. In particular, it is of the form
    \begin{equation}
        \label{eq:transitionKernelAlpha+rho}
        q(x,\cdot)=\alpha(x,\cdot)+\rho(x)\delta_x(\cdot),\quad \nu\text{-a.s.},
    \end{equation}
    with $\alpha(x,\cdot)\ll a(x,\cdot)$ $\nu$-a.s. and $\rho(x)$ is a measurable function.
\end{lemma}

\begin{proof}
    If $I(\nu)<\infty$, then \eqref{eq:rateFunctionIntegral} implies $R(q(x,\cdot)\parallel K(x,\cdot))<\infty$ $\nu-$a.s. By the definition of relative entropy, this means that $q(x,\cdot)\ll K(x,\cdot)$ $\nu-$a.s. Recall that
    \begin{equation*}
        K(x,dy)=a(x,y)dy+r(x)\delta_x(dy),
    \end{equation*}
    i.e. $K(x,\cdot)$ is a mixture of a transition kernel $a(x,\cdot)\ll\lambda$, and a point mass in $x$. Therefore, for the transition kernel $q(x,\cdot)$ to be $q(x,\cdot)\ll K(x,\cdot)$ $\nu-$a.s., it must be of the form
    \begin{equation*}
    q(x,y)=\alpha(x,y)dy+\rho(x)\delta_x(dy),
    \end{equation*}
    where $\alpha(x,\cdot)\ll a(x,\cdot) \ll\lambda$, and $\rho(x)=0$ if $r(x)=0$. In particular, $\rho(x)$ must be a measurable function in order to make $x\mapsto q(x,A)$ a measurable function for every $A\in\mathcal{B}(S)$, and therefore $q$ a stochastic kernel.
\end{proof}

With the characterization of $q$ from Lemma~\ref{lem:transitionKernelForm}, we can write the rate function \eqref{eq:rateFunctionIntegral} in a more explicit way.
\begin{proposition}
\label{prop:decomositionOfRateFunctionAlphaRho}
    If $I(\nu)<\infty$, then the rate function can be expressed as
    \begin{equation}
        \label{eq:rateFuncDecomposed}
        I(\nu)=\int_S \int_S\log\left(\frac{\alpha(x,y)}{a(x,y)}\right)\alpha(x,y)\,dy\,\nu(dx)+\int_S\log\left(\frac{\rho(x)}{r(x)}\right)\rho(x)\,\nu(dx),
    \end{equation}
    with $\alpha(x,y)$ and $\rho(x)$ as in Lemma~\ref{lem:transitionKernelForm}.
\end{proposition}
\begin{proof}
    Applying the definition of relative entropy in \eqref{eq:rateFunctionIntegral}, the rate function becomes
    \begin{equation}
        \label{eq:rateFuncWithRadonNikodym}
        I(\nu)=\int_S\int_S\log\left(f_x(y)\right)q(x,dy)\nu(dx),
    \end{equation}
    where $f_x$ denotes the Radon-Nikodym derivative of the transition kernel $q(x,\cdot)$ with respect to $K(x,\cdot)$ for a fixed $x\in S$. By Lemma~\ref{lem:transitionKernelForm} $f_x$ exists $\nu$-a.s.\ and, by combining \eqref{eq:MHtransitionKernel} and \eqref{eq:transitionKernelAlpha+rho}, 
    \begin{equation}
        \label{eq:RNderivative}
        f_x(y)=\frac{\alpha(x,y)}{a(x,y)} I\{y\neq x\} + \frac{\rho(x)}{r(x)} I\{y=x\}.
    \end{equation}
    Indeed, let $A\in\mathcal{B}(S)$ and recall that $a(x,\cdot)\ll\lambda$. Then, it holds
    \begin{align*}
        \int_A f_x(y)K(x,dy)&=\int_A\left(\frac{\alpha(x,y)}{a(x,y)} I\{y\neq x\} + \frac{\rho(x)}{r(x)} I\{y=x\}\right)\left(a(x,y)dy+r(x)\delta_x(dy)\right)\\
        &=\int_A\alpha(x,y)dy+\rho(x)\delta_x(A)=q(x,A),
    \end{align*}
    for $\nu$-almost all $x\in S$. This proves that $f_x(y)$ in~\eqref{eq:RNderivative} is the Radon-Nikodym derivative of $q(x,\cdot)$ with respect to $K(x,\cdot)$ for $\nu$-almost all $x$ in $S$.
    
    Replacing $f_x(y)$ in \eqref{eq:rateFuncWithRadonNikodym} with \eqref{eq:RNderivative} gives
    \begin{align*}
        I(\nu)&=\int_S\int_S\log\left(\frac{\alpha(x,y)}{a(x,y)} I\{y\neq x\} + \frac{\rho(x)}{r(x)} I\{y=x\}\right)\left(\alpha(x,y)dy+\rho(x)\delta_x(dy)\right)\nu(dx)\\
        &=\int_S\left(\int_S\log\left(\frac{\alpha(x,y)}{a(x,y)}\right)\alpha(x,y)dy+\log\left(\frac{\rho(x)}{r(x)}\right)\rho(x)\right)\nu(dx),
    \end{align*}
    which leads to \eqref{eq:rateFuncDecomposed}.
\end{proof}
We end this section with an alternative characterization of the rate function, that highlights the fact that measures $\nu \in \calP (S)$ for which $I(\nu) < \infty$ need not be absolutely continuous with respect to $\pi$.

For any $\nu \in \calP (S)$, by the Lebesgue decomposition theorem, we have
    \begin{equation}
        \label{eq:decomposition}
        \nu = (1-p)\cdot\nu_\lambda+p\cdot\nu_s,
    \end{equation}
    where $p\in[0,1]$,  $\nu_\lambda,\nu_s\in\mathcal{P}(S)$, with $\nu_\lambda\ll\lambda$ and $\nu_s\perp\lambda$. Note that $p$ is specific to $\nu$, which we suppress in the notation. Associated with the decomposition \eqref{eq:decomposition}, we also define the partition $S=S_\lambda\cup S_s$, with $S_s \cap S_\lambda = \emptyset$, $\nu_s(S_\lambda)=0$ and $\lambda(S_s)=0$. The following Lemma shows that $I(\nu)$ is split into two parts, one corresponding to $\nu _\lambda$ and one corresponding to $\nu _s$.
\begin{lemma}
\label{lem:nuLambdaAndSingularInvariantForQ}
Let $\nu\in\mathcal{P}(S)$ with $I(\nu)<\infty$ and consider its decomposition as in \eqref{eq:decomposition}. Let $q(x,dy)$ be a transition kernel on $S$ with invariant distribution $\nu$, that satisfies 
\[I(\nu)=\int_SR(q(x,\cdot)\parallel K(x,\cdot))\nu(dx).\]
 Define $\mathcal{Q}_\lambda$ and $\mathcal{Q}_s$ as the set of transitions kernels that $\nu _\lambda$ and $\nu _s$ are invariant 
for, respectively. The following holds:
    \begin{enumerate}[label=(\alph*)]
        \item \label{itm:qInvariant} $q\in\mathcal{Q}_\lambda\cap\mathcal{Q}_s$, i.e. both $\nu_\lambda$ and $\nu_s$ are invariant for $q$,
        \item \label{itm:rateFcnDecomposition} the rate function satisfies 
        \begin{equation}
        \label{eq:decompositionOfRateFunction}
        I(\nu) = (1-p) I(\nu _\lambda) + p I(\nu _s).
        \end{equation}
    \end{enumerate}
\end{lemma}

\begin{proof}
    \ref{itm:qInvariant} By Lemma~\ref{lem:transitionKernelForm}, we can write
\[q(x,\cdot)=\alpha(x,\cdot)+\rho(x)\delta_x(\cdot),\quad \nu\text{-a.s.},\]
where $\alpha(x,\cdot)\ll\lambda$. By invariance of $\nu$ for $q$, for all $A\in\mathcal{B}(S)$,
    \[\nu(A)=\int_Sq(x,A)\nu(dx)=\int_S\alpha(x,A)\nu(dx)+\int_A\rho(x)\nu(dx).\]
    If we consider $A=S_s$, for which $\lambda(S_s)=0$, then $\alpha(x,S_s)=0$ for $\nu$-almost all $x\in S$ (because of $\alpha(x,\cdot)\ll\lambda$), and thus $\nu(S_s)=\int_{S_s}\rho(x)\nu(dx).$ On the other hand, $\nu(S_s)=\int_{S_s}\nu(dx)$. This implies that for all $x\in S_s$ $\nu$-a.s., we have that $\rho(x)=1$ a.s., and therefore $q(x,dy)=\delta_x(dy)$.

    With the form of $q$ on $S_s$ established, for $A\in\mathcal{B}(S)$, we have
        \begin{equation*}
        \int_Sq(x,A)\nu_s(dx)=\int_{S_s}q(x,A)\nu_s(dx)=\int_{S_s}\delta_x(A)\nu_s(dx)=\nu_s(A),
    \end{equation*}
    where the last equality is due to $q(x,\cdot)=\delta_x(\cdot)$ a.s.\ being the only $\nu_s$-invariant transition kernel (Lemma~\ref{lem:deltaIsKernelForSingularMeasure}). This proves that $\nu_s$ is invariant for $q$, which means that $q\in\mathcal{Q}_s$.

    We now show that $\nu_\lambda$ is also invariant for $q$. The decomposition \eqref{eq:decomposition} combined with the invariance of $\nu$ for $q$, and given that $q(x,\cdot)=\delta_x(dy), \nu_s$-a.s., gives, for $A\in\mathcal{B}(S)$,
    \begin{align*}
        (1-p)\cdot\nu_\lambda(A)+p\cdot\nu_s(A)&=\nu(A)=\int q(x,A)\nu(dx)\\
        &=(1-p)\cdot\int q(x,A)\nu_\lambda(dx)+p\cdot\int q(x,A)\nu_s(dx)\\
        &=(1-p)\cdot\int q(x,A)\nu_\lambda(dx)+p\cdot\nu_s(A).
    \end{align*}
    It follows that $\nu_\lambda(A)=\int q(x,A)\nu_\lambda(dx)$. Since $A\in\mathcal{B}(S)$ was chosen arbitrarily, $\nu_\lambda$ is invariant for $q$, i.e., $q\in\mathcal{Q}_\lambda$.
    
   To prove \ref{itm:rateFcnDecomposition}, by convexity of $I$ (see Lemma~6.10(a) in \cite{BudhirajaAmarjit2019AaAo}),

    \begin{equation}
    \label{eq:ConveityRateFunctionLe}
        I(\nu)=I\left((1-p)\cdot\nu_\lambda+p\cdot\nu_s\right)\le (1-p)\cdot I(\nu_\lambda)+p\cdot I(\nu_s).
    \end{equation}
    On the other hand, by the decomposition \eqref{eq:decomposition},
    \begin{align}
    I(\nu)&=\int_SR(q(x,\cdot)\parallel K(x,\cdot))\nu(dx) \label{eq:rateFuncGe}\\
        &= (1-p)\cdot \int_SR(q(x,\cdot)\parallel K(x,\cdot))\nu_\lambda(dx)+p\cdot \int_SR(q(x,\cdot)\parallel K(x,\cdot))\nu_s(dx)\nonumber.
    \end{align}
    From part (a), $q$ is an element of both $\mathcal{Q}_\lambda$ and $\mathcal{Q}_s$. Therefore,
    \begin{equation*}
        \int_SR(q(x,\cdot)\parallel K(x,\cdot))\nu_\lambda(dx)\ge \inf_{\tilde q\in \mathcal{Q}_\lambda} \int_SR(\tilde q(x,\cdot)\parallel K(x,\cdot))\nu_\lambda(dx).
    \end{equation*}
    The right-hand side of the previous display is precisely $I(\nu _\lambda)$. Similarly,
        \begin{equation*}
        \int_SR(q(x,\cdot)\parallel K(x,\cdot))\nu_s(dx)\ge \inf_{\tilde q\in \mathcal{Q}_s} \int_SR(\tilde q(x,\cdot)\parallel K(x,\cdot))\nu_s(dx),
    \end{equation*}
    and the right-hand side of this inequality is now $I(\nu _s)$. The two inequalities together with \eqref{eq:rateFuncGe} imply
    \begin{equation*}
        I(\nu)\ge (1-p)\cdot I (\nu_\lambda) + p\cdot I(\nu_s).
    \end{equation*}
    Combined with the opposite inequality \eqref{eq:ConveityRateFunctionLe}, this proves the desired equality \eqref{eq:decompositionOfRateFunction}.

\end{proof}

\section{Laplace principle upper bound}
\label{sec:upper}
In this section we prove the Laplace principle upper bound \eqref{eq:upper}.
\begin{proposition}
    \label{lemma:upper}
    Let $\{L^n\}_{n\ge 0}$ be the empirical measures defined in~\eqref{eq:empiricalMeasure} and $\{x_n\}_{n\ge 0}$ any sequence in $S$. Take $F \in C_b \left(\calP(S) \right)$ and define $I:\mathcal{P}(S)\to[0,\infty]$ as in~\eqref{eq:rateFunction}. Assume~\ref{ass:targetAbsContLambda}, \ref{ass:proposalDistributionAbsCont} and~\ref{ass:compactSpaceOrLyapunov}. Then, 
    \begin{equation*}
        \liminf_{n\to\infty}-\frac{1}{n}\log \mathbb{E}_{x_n}\left[ e^{-nF(L^n)} \right]\ge \inf_{\nu\in\mathcal{P}(S)}[F(\nu)+I(\nu)].
    \end{equation*}
\end{proposition}
As mentioned in Section \ref{sec:LDPMH}, under \ref{ass:targetAbsContLambda}-\ref{ass:compactSpaceOrLyapunov}, the arguments from \cite{BudhirajaAmarjit2019AaAo} can be used. We include the main steps here for self-containment and convenience of the reader; we emphasise that once the Feller property of $K(x, dy)$ has been established, this part of the proof goes precisely as in \cite{BudhirajaAmarjit2019AaAo}. 

\begin{lemma}
    \label{lem:MHKernelFeller}
    Under Assumptions~\ref{ass:targetAbsContLambda}-\ref{ass:proposalDistributionAbsCont}, the Metropolis-Hastings transition kernel $K(x,\cdot)$ satisfies the Feller property. 
\end{lemma}
\begin{proof}
Recall the form \eqref{eq:MHtransitionKernel} for $K$, with $a(x,y)$ in \eqref{eq:acceptanceDensity} corresponding to the probability density of the acceptance part and $r$ corresponding to the rejection part.  The assumptions ensure that both $a$ and $r$ are continuous (see Remark \ref{rmk:continuityAandR}) and bounded as functions of $x$.  Consider now a function $f \in C_b (S)$, and a sequence $\{x_n\}_{n\in\mathbb{N}}\subset S$ such that $x_n\to x\in S$. By dominated convergence, we have
    \begin{align*}
        \int_Sf(y)K(x_n,dy)&=\int_Sf(y)a(x_n,y)dy+f(x_n)r(x_n)\\
        &\to \int_Sf(y)a(x,y)dy+f(x)r(x)=\int_Sf(y)K(x,dy).
    \end{align*}
   An application of the Portmanteau theorem then completes the proof.

\end{proof}

\begin{proof}[Proof of Proposition \ref{lemma:upper}]
In \eqref{eq:rep}, take a control sequence $\{ \bar \mu ^n _i \}$ such that 
\begin{align*}
&\bE \left[ F(\bar L ^n) + \frac{1}{n} \sum_{i=1} ^n R(\bar \mu ^n _i \parallel K(\bar X ^n _i, \cdot)) \right] \\
&\qquad \leq 
 \inf _{\{ \hat \mu ^n _i \} } \bE \left[ F(\hat L ^n) + \frac{1}{n} \sum_{i=1} ^n R(\hat \mu ^n _i \parallel K(\hat X ^n _i, \cdot)) \right] + \frac{1}{n},
\end{align*}
where $\bar L ^n$ is the controlled empirical measure associated with $\{ \bar \mu ^n _i\}$. Let 
\[
    \lambda ^n (dx \times dy) = \frac{1}{n} \sum _{i=1} ^n \delta _{\bar X ^n _{i-1}}(dx) \bar \mu ^n _i (dy). 
\]
By Assumption \ref{ass:compactSpaceOrLyapunov}, $\{ ( L ^n, \lambda ^n) \}$ is tight; see Section 10 in \cite{BudhirajaAmarjit2019AaAo}. Thus, there is a subsequence, also denoted by $n$, such that $\{ ( L ^n, \lambda ^n) \}$ converge along that subsequence, to some limit $(\bar L, \lambda)$, and it is enough to prove the upper bound \eqref{eq:upper} for this subsequence. In fact, taking $n \to \infty$, we have
\begin{align*}
\liminf _{n \to \infty} - \frac{1}{n} \bE _{x_n} \left[ e^{-n F(\bar L ^n)} \right]  &\geq \bE \left[ F(\bar L) + R(\lambda \parallel \bar L \otimes K) \right]  \\
& \geq \inf _{\nu \in \calP (S) } \left[ F(\nu) + \inf _{\gamma \in A (\nu)} R(\gamma \parallel \nu \otimes K ) \right] \\
&= \inf _{\nu \in \calP (S) } \left[ F(\nu) +  I(\nu) \right].
\end{align*}
 \end{proof}

\section{Laplace principle lower bound}
\label{sec:lower}
We now proceed to prove the Laplace principle lower bound \eqref{eq:lower}.
\begin{proposition}
    \label{lemma:lower}
    Let $\{L^n\}_{n\ge 0}$ be the empirical measures defined in~\eqref{eq:empiricalMeasure} and define $I:\mathcal{P}(S)\to[0,\infty]$ as in~\eqref{eq:rateFunction}. Assume~\ref{ass:targetAbsContLambda}-\ref{ass:proposalDistributionAbsCont}. Then, for $x \in S$,
    \begin{equation}
    \label{ineq:LaplaceLowerBound}
        \limsup_{n\to\infty}-\frac{1}{n}\log \mathbb{E}_{x}\left[e^{-nF(L^n)}\right]\le \inf_{\nu\in\mathcal{P}(S)}[F(\nu)+I(\nu)].
    \end{equation}
\end{proposition}
As described in Section \ref{sec:LDPMH}, in proving Theorem \ref{thm:LDP}, the lower bound is where the lack of Condition \ref{cond:DE} for the MH kernel $K$ plays a role. To see why the lack of this transitivity condition becomes an issue, one of the consequences of the condition is that if $\nu \in \calP (S)$ is such that $I(\nu) < \infty$, then $\nu \ll \pi$. This property plays an important role in the proof of the LDP for empirical measures of a Markov chain in \cite{BudhirajaAmarjit2019AaAo}—it is implicitly used to define a sequence of near-optimal controls in the representation \eqref{eq:rep}. Here, because of the rejection part of the MH kernel, which is the reason Condition \ref{cond:DE} does not hold, the implication is not true in general. As a counterexample, consider an $x_0 \in S$ such that $r(x_0) > 0$ and take $\nu = \delta _{x_0} \in \calP (S)$. Then $\nu$ is not absolutely continuous with respect to $\lambda$, and thus not with respect to $\pi$. Consider the transition kernel $\tilde q (x, \cdot) = \delta _{x}$. Then $\nu$ is invariant for $\tilde q$ and from \eqref{eq:rateFuncAsInf},
\begin{equation*}
    I(\nu)\le  \int_S R(\delta_{x}(\cdot)\parallel K(x,\cdot))\nu(dx) = R(\delta_{x_0}(\cdot)\parallel K(x_0,\cdot)).
\end{equation*}
From \eqref{eq:RNderivative}, the Radon-Nikodym derivative of $\delta_{x_0}(\cdot)$ with respect to $K(x,\cdot)$, for $x=x_0$, is given by $f_{x_0}(y)=\frac{1}{r(x_0)}I\{y=x_0\}$.  It follows that the rate function is finite, since
\begin{equation*}
    I(\nu)\le R(\delta_{x_0}(\cdot)\parallel K(x_0,\cdot)) \le \log\frac{1}{r(x_0)}<\infty.
\end{equation*} 

We circumvent the problem of not having Condition \ref{cond:DE} by showing that if $\nu\in\mathcal{P}(S)$ is such that $I(\nu)<\infty$, then there exists another probability measure $\nu^*\in\mathcal{P}(S)$ that is arbitrarily close to $\nu$, and satisfies $\nu ^\star \ll \pi$ and $I(\nu^*)\le I(\nu)+\varepsilon$. 

To prove the existence of such a measure, recall that the decomposition \eqref{eq:decomposition} allows us to separate $\nu$ into two parts: one part, $\nu _\lambda$, with a density with respect to $\lambda$ (and thus with respect to $\pi$) and one, $\nu _s$, that is singular with respect to $\lambda$. The idea is to approximate the latter with measures that are absolutely continuous with respect to $\lambda$. This allows us to construct near-optimal controls in the representation formula, which in turn are used to prove Proposition \ref{lemma:lower}. 

The following is a brief outline of the argument.

In Lemma \ref{lem:deltaIsKernelForSingularMeasure}, we characterize the transition kernels $q$ that achieve the infimum in \eqref{eq:rateFuncAsInf} for $\nu _s \in \calP (S)$ such that $\nu _s \perp \lambda$ and $I(\nu _s) < \infty$. Next, in Lemma \ref{lem:construction}, we construct a sequence of random measures $\{ \nu ^n _s \} \subset \calP (S)$ that are absolutely continuous with respect to $\lambda$, $\nu ^n _s \Rightarrow \nu$ as $n \to \infty$, and $I(\nu ^n _s) \to I(\nu)$. This construction allows us to show (Lemma \ref{lem:existenceOfCloseMeasure}) that for any $\nu \in \calP (S)$ such that $I(\nu) < \infty$, for any $\varepsilon > 0$, there exists a $\nu ^\dagger \in \calP(S)$ that is arbitrarily close to $\nu$, $\nu ^\dagger \ll \lambda$ and $I(\nu ^\dagger) \leq I(\nu) + \varepsilon$. The existence of such a probability is then used in  Lemma~\ref{lem:existenceOfNu*} to prove the existence of a $\nu ^* \in \calP (S)$ with the desired properties. From there, the proof of Proposition \ref{lemma:lower} follows largely that of \cite{BudhirajaAmarjit2019AaAo}.
   
\begin{lemma}
\label{lem:deltaIsKernelForSingularMeasure}
Let $\nu_s \in \calP(S)$ be such that $\nu _s \perp \lambda$ and $I(\nu_s)<\infty$. Then, $q(x,\cdot)=\delta_x(\cdot)$ $\nu_s-$a.s.\ is the only transition kernel that satisfies \eqref{eq:rateFunctionIntegral}, i.e.,
\begin{equation*}
    I(\nu_s)=\int_SR(\delta_x(\cdot)\parallel K(x,\cdot))\nu_s(dx) = \int_S\log\frac{1}{r(x)}\nu_s(dx) 
\end{equation*}

\end{lemma}
\begin{proof}
By Lemma~\ref{lem:transitionKernelForm}, if $I(\nu_s)<\infty$, then the kernels $q(x,\cdot)$ that satisfy \eqref{eq:rateFunctionIntegral} are of the form $\alpha(x,\cdot)+\rho(x)\delta_x(\cdot)$, $\nu_s-$a.s. with $\alpha(x,\cdot)\ll a(x,\cdot)$,  $\nu_s-$a.s. Moreover, $a(x,\cdot)$ is in turn absolutely continuous with respect to $\lambda$. Observe that since the set $S_s$ satisfies $\lambda(S_s)=0$, then $a(x,S_s)=0$ and therefore $\alpha(x,S_s)=0$. On the other hand, $\nu_s(S_s)=1$ by definition, and by invariance the following must hold:
\begin{align*}
    1&=\nu_s(S_s) \\
    &=\int_Sq(x,S_s)\nu_s(dx) \\
    &=\int_S\left(\alpha(x,S_s)+\rho(x)\delta_x(S_s)\right)\nu_s(dx)\\
    &=\int_S\left(0+\rho(x)\delta_x(S_s)\right)\nu_s(dx)\\ &=\int_{S_s}\rho(x)\nu_s(dx).
\end{align*}
Given that $\rho(x)\le 1 \;\forall x\in S$, $\int _{S_s} \rho(x) \nu _s (dx) = 1$ can only hold if $\rho(x)\equiv 1$ $\nu_s-$a.s. We conclude that the singular measure $\nu_s$ admits only $q(x,\cdot)=\delta_x(\cdot)$ $\nu_s-$a.s.\ as invariant kernel. This implies that
\[
 I(\nu_s)=\int_SR(\delta_x(\cdot)\parallel K(x,\cdot))\nu_s(dx).
 \]
Furthermore, by Proposition \ref{prop:decomositionOfRateFunctionAlphaRho}, we have
\begin{equation*}
   \int_SR(\delta_x(\cdot)\parallel K(x,\cdot))\nu_s(dx) = \int_S\log\frac{1}{r(x)}\nu_s(dx).
\end{equation*}
This completes the proof.
\end{proof}

We now move to the construction of a sequence of random measures $\{ \nu ^n _s \} \subset \calP (S)$ that can be used to approximate $\nu _s$  arbitrarily well and satisfy $\lim _{n \to \infty} I(\nu ^n _s) \leq I(\nu)$ a.s., while maintaining absolute continuity with respect to $\lambda$. To facilitate this, we define, for $\varepsilon > 0$ and $x \in S_+$, 
\begin{align}
    \Delta_\epsilon(x)= \sup\{t \,:\,&\lvert\log a(x,x) - \log a(y,z)\rvert < \epsilon \text{ and }\nonumber \\
    &\lvert\log r(x) - \log r(y)\rvert < \epsilon, \qquad\forall y,z\in B_t(x)\}.    \label{def:deltaEpsilon}
    \end{align}
    
\begin{lemma}
\label{lem:construction}
Take $\nu _s \in \calP (S)$ such that $\nu _s \perp \lambda$ and $I(\nu _s) < \infty$. Let $\{Y_i\}_{i=1}^\infty$ be independent and identically distributed according to $\nu_s$. 
    For $n \in \bN$, define 
    \begin{equation}
    \label{def:varrhon}
        \varrho^n = \min\left\{ \frac{1}{n}, \min_{1\le i\le n}\Delta_{\frac{1}{n}}(Y_i), \frac{1}{2}\min_{Y_i\neq Y_j}d_S(Y_i,Y_j), \frac{1}{2}\min_{1\le i \le n}d_S(\partial S, Y_i), \min_{1\le i \le n}a(Y_i,Y_i)\right\}.
    \end{equation}
    Let $V_n= \lambda(B_{\varrho^n}(0))$, the (Lebesgue) volume of the balls of radius $\varrho^n$, and define the sequence of random measures $\{\nu_s^n\}_{n \in \bN}\subset\mathcal{P}(S)$ by
    \begin{equation}
    \label{def:nu_s_n}
        \nu_s^n(dx):=\frac{1}{n}\frac{1}{V_n}\sum_{i=1}^nI\{x\in B_{\varrho^n}(Y_i)\}\lambda(dx).
    \end{equation}

    This sequence satisfies the following properties:
    \begin{enumerate}[label=(\alph*)]
    \item \label{item:abscont} $\nu_s^n\ll\lambda$ for all $n\in\mathbb{N}$,
    \item \label{item:weakconv} $\nu_s^n\Rightarrow \nu_s$ a.s.,
    \item \label{item:finiteRate} There is an $n_0 \in \bN$ such that, for all $n > n_0$, $I(\nu_s^n)<\infty$ a.s., 
    \item \label{item:convRate} $\lim_{n\to\infty}I(\nu_s^n)\le\ I(\nu_s)$ a.s.
\end{enumerate}
\end{lemma}
Before we embark on the proof, some comments on the construction. 
First, because we consider $\nu _s$ such that $I (\nu _s) < \infty$, $\nu _s$ can only put mass on points in $S_+$: if $\nu _s (x) > 0$ for some $x$ such that $\pi (x) = 0$, then $r(x) = 0$ (see Remark \ref{rmk:continuityAandR}). By Lemma \ref{lem:transitionKernelForm}, the corresponding transition kernel is of the form $q(x,\cdot) = \alpha (x, \cdot)$, where $\alpha(x,\cdot) \ll a(x,\cdot)$. This is not compatible with $\nu _s$ being singular with respect to $\lambda$; see also Lemma \ref{lem:deltaIsKernelForSingularMeasure}. Thus, the $Y_i$s used in the construction are in $S_+$ $\nu _s$-a.s. 

Next, we verify that for any fixed $n \in \bN$, the radius $\varrho ^n$ of the $B _{\varrho ^n} (Y_i)$-balls is well-defined, i.e., $\varrho ^n > 0$ $\nu _s$-a.s. Note that if $\nu _s = \delta _{x}$ for some $x \in S_+$, then $\varrho ^n$ becomes 
\begin{equation*}
        \varrho^n = \min\left\{ \frac{1}{n}, \Delta_{\frac{1}{n}}(x), \frac{1}{2}d_S(\partial S, x), a(x,x)\right\}.
    \end{equation*}

Because $I(\nu _s) < \infty$, we have for $Y_i \sim \nu _s$,
    \[
        \bE \left[ \log \frac{1}{r(Y_i)} \right] = \int _S \log \frac{1}{r(x)} \nu _s(dx) = I(\nu _s) < \infty.
    \]
    It follows that $r(Y_i) > 0 $ w.p.\ 1. From Assumption~\ref{ass:proposalDistributionAbsCont} we have $a(Y_i,Y_i)=J(Y_i|Y_i)>0$. Since the support of $\nu _s$ is in $S_+$ (an open subset of $S$; see Assumption~\ref{ass:targetAbsContLambda}), and $a(Y_i, Y_i)$ and $r(Y_i)$ are both strictly positive $\nu _s$-a.s., the continuity of $r(\cdot)$ and $a(\cdot,\cdot)$ on $S$ and $S_+\times S$, respectively, ensures that $\Delta_{\frac{1}{n}}(Y_i)>0$, $i=1, \dots, n$. Moreover, $d_S(Y_i,Y_j)>0$ for $Y_i\neq Y_j$ by definition, and $d_S(\partial S,Y_i)>0$ $\nu _s$-a.s.\ since the support of $\nu_s$ is a subset of $S_+$, which is an open subset of $S$. 
    Combined, these show that $\varrho^n>0$ $\nu_s$-a.s.

\begin{proof}[Proof of Lemma \ref{lem:construction}]
Part \ref{item:abscont} follows directly from the definition \eqref{def:nu_s_n} of $\nu_s^n$. 
In particular, since $\lambda$ and $\pi$ are equivalent measures (Assumption~\ref{ass:targetAbsContLambda}), then $\nu_s^n\ll\pi$.

To prove \ref{item:weakconv}, that the sequence $\{\nu_s^n\}$ converges weakly to $\nu_s$ a.s., we show that for any bounded and Lipschitz continuous function $f$ it holds that $\int_Sfd\nu_s^n\to\int_Sfd\nu_s$ a.s. An application of the Portmanteau theorem then gives the claim.
    
    To this end, let $f \in C_b(S)$ be Lipschitz continuous and denote its Lipschitz constant by $L_f<\infty$. For $n \in \bN$, we have
\begin{equation}
\label{eq:sequenceIntegrals}
    \int_S f(x) \nu_s^n (dx)  = \frac{1}{n}\frac{1}{V_n}\sum_{i=1}^n\int_{B_{\varrho^n}(Y_i)}f(x) \lambda(dx).
\end{equation}

The Lipschitz continuity of $f$ implies that for all $x\in B_{\varrho^n}(Y_i)$ and for all $i \in \{ 1, \dots, n\}$,
\begin{equation*}
    f(Y_i) - L_f\cdot \varrho^n \le f(x) \le f(Y_i) + L_f\cdot \varrho^n.
\end{equation*}
 
By integrating over $B_{\varrho^n}(Y_i)$ and dividing by $V_n$, it follows that 
\begin{equation*}
    f(Y_i) - L_f\cdot \varrho^n \le \frac{1}{V_n} \int_{B_{\varrho^n}(Y_i)} f(x) \lambda(dx) \le f(Y_i) + L_f\cdot \varrho^n, \ \ i =1, \dots, n.
\end{equation*}
This implies the following bounds on the integral \eqref{eq:sequenceIntegrals}:
\begin{align*}
     \frac{1}{n}\sum_{i=1}^nf(Y_i)-\frac{L_f\cdot \varrho^n}{n}&\le\int_Sf(x) \nu_s^n (dx) \le\frac{1}{n}\sum_{i=1}^nf(Y_i)+\frac{L_f\cdot \varrho^n}{n}.
\end{align*}
By the strong law of large numbers, $\frac{1}{n}\sum_{i=1}^n\delta_{Y_i}(\cdot)\Rightarrow\nu_s(\cdot)$ a.s., and it follows that $\frac{1}{n}\sum_{i=1}^nf(Y_i)\to \int_S f d\nu_s$ a.s. Moreover, by construction $\varrho^n\to 0$ as $n \to \infty$, which implies $\frac{L_f\cdot \varrho^n}{n}\to 0$. The squeeze theorem now yields the desired result.

We now move to part \ref{item:finiteRate}. To show that $I(\nu_s^n)$ is finite for large enough $n\in\mathbb{N}$, we first note that by construction, $V_n\to0$ as $n \to \infty$. Therefore, there is an $n_0 \in \bN$ such that $V_n < 1$ for all $n > n_0$. Henceforth, we only consider such $n$. 

Recall the characterization \eqref{eq:rateFuncAsInf} of the rate function,
\begin{equation*}
    I(\nu_s^n)=\inf_{q} \int_SR(q(x,\cdot)\parallel K(x,\cdot))\nu_s^n(dx),
\end{equation*}
where the infimum is taken over all the transition kernels $q(x,dy)$ that are $\nu_s^n-$irreducible. We will now construct such a transition kernel $q^n(x,dy)$, for which it also holds that
\begin{equation*}
    \int_SR(q^n(x,\cdot)\parallel K(x,\cdot))\nu_s^n(dx) < \infty.
\end{equation*}
This in turn implies that $I(\nu_s^n)<\infty$. The collection of transition kernels $\{ q^n \}$ will also be used to show part \ref{item:convRate}. 

We begin by defining $N^n(x)$ as the number of $B_{\varrho^n}(Y_i)$, $i=1, \dots, n$, that $x \in S$ belongs to,
\begin{equation*}
    N^n(x) = \sum_{i=1}^n I\{x\in B_{\varrho^n}(Y_i)\}. %
\end{equation*}

Next, we define $q^n$ by
\begin{align*}
q^n(x,dy) = \frac{1}{N^n(x)}\sum_{i=1}^nI\{x\in B_{\varrho^n}(Y_i)\}I\{y\in B_{\varrho^n}(Y_i)\}dy+ \left(1-V_n\right)\delta_x(dy),
\end{align*}
for $x$ such that $N^n(x) \geq 1$, and otherwise $q^n(x,dy) = \delta _x (dy)$. Then, for all $x \in S$, $q^n(x,\cdot)$ is a transition probability: if $N^n(x) \geq 1$,
\begin{equation*}
    q^n(x,S)=\frac{1}{N^n(x)}\sum_{i=1}^nI\{x\in B_{\varrho^n}(Y_i)\}V_n + \left(1-V_n\right)\delta_x(S)=1,
\end{equation*}
and, for $N^n(x) = 0$, it holds immediately that $q^n (x,S) = 1$. Moreover, due to the choice of $n > n_0$ $q^n(x,A) \in [0,1]$, for every $A \in \calB (S)$. 

To show that $q^n(x,\cdot)$ is also invariant for $\nu_s^n$, consider a set $A\in\mathcal{B}(S)$. We have 
\begin{equation*}
    \nu_s^n(A) = \frac{1}{n}\frac{1}{V_n}\sum_{i=1}^n\int_AI\{x\in B_{\varrho^n}(Y_i)\} \lambda(dx)=\frac{1}{n}\frac{1}{V_n}\sum_{i=1}^n\lambda(A\cap B_{\varrho^n}(Y_i)).
\end{equation*}
Take $x\in S$. If $N^n(x) \geq 1$,
\begin{align*}
    &q^n(x,A) \\
    &=\frac{1}{N^n(x)}\sum_{i=1}^nI\{x\in B_{\varrho^n}(Y_i)\}\int_AI\{y\in B_{\varrho^n}(Y_i)\}\lambda(dy) + \left(1-
    V_n\right)\delta_x(A)\\
    &=\frac{1}{N^n(x)}\sum_{i=1}^nI\{x\in B_{\varrho^n}(Y_i)\}\lambda(A\cap B_{\varrho^n}(Y_i)) + \left(1-V_n\right)\delta_x(A).
\end{align*}
From this it follows that 
\begin{align*}
&\int_Sq^n(x,A)d\nu_s^n(dx)\\
&=\frac{1}{n}\frac{1}{V_n}\sum_{i=1}^n\int_{B_{\varrho^n}(Y_i)}\left(\frac{1}{N^n(x)}\sum_{j=1}^nI\{x\in B_{\varrho^n}(Y_j)\}\lambda(A\cap B_{\varrho^n}(Y_j)) + \left(1-V_n\right)\delta_x(A)\right)\lambda(dx)\\
&=\frac{1}{n}\frac{1}{V_n}\sum_{i=1}^n\int_{B_{\varrho^n}(Y_i)}\left(\lambda(A\cap B_{\varrho^n}(Y_i))+(1-V_n)\lambda(A\cap B_{\varrho^n}(Y_i))\right)\lambda(dx)\\
&=\frac{1}{n}\frac{1}{V_n}\sum_{i=1}^n\lambda(A\cap B_{\varrho^n}(Y_i))=\nu_s^n(A),
\end{align*}
where in the second equality we use that, due to the definition of $\varrho ^n$, there are no overlaps between the $B _{\varrho ^n} (Y_i)$-balls. If instead $N^n(x) = 0$, then $q^n(x,A) = \delta _x (A)$, and we have 
\[
    \int _S \delta _x (A) \nu ^n _s (dx) = \int _A \nu ^n _s (dx) = \nu ^n _s (A).
\]
Combined with the computation for $N^n(x) \geq 1$, this proves the invariance. 

From \eqref{eq:rateFuncAsInf}, $I(\nu ^n _s)$ is defined in terms of the infimum over the set of $\nu_s^n$-invariant kernels \eqref{eq:rateFuncAsInf}. Therefore,
\begin{align}
I(\nu_s^n)&\le\int_SR(q^n(x,\cdot)\parallel K(x,\cdot))\nu_s^n(dx)\nonumber\\
&=\int _{\{x: N^n (x) = 0\}} R(q^n(x,\cdot)\parallel K(x,\cdot))\nu_s^n(dx) \nonumber \\ 
&\qquad + \int _{\{x: N^n (x) \geq 1\}} R(q^n(x,\cdot)\parallel K(x,\cdot))\nu_s^n(dx). \nonumber
\end{align}
For the first integral in the last display, since $\nu^n _s$ has no mass on $\{ x\in S : N^n(x) = 0 \}$, this integral is zero. For the second integral, we have
\begin{align*}
& \int _{\{x: N^n (x) \geq 1\}} R(q^n(x,\cdot)\parallel K(x,\cdot))\nu_s^n(dx) \\
& \quad =\frac{1}{n}\frac{1}{V_n}\sum_{i=1}^n\int_{B_{\varrho^n}(Y_i)}\left(\int_{B_{\varrho^n}(Y_i)}\log\frac{1}{a(x,y)}\lambda(dy)+(1-V_n)\cdot\log\frac{1-V_n}{r(x)}\right)\lambda(dx)\nonumber\\
& \quad =\frac{1}{n}\frac{1}{V_n}\sum_{i=1}^n\left(\iint_{\left(B_{\varrho^n}(Y_i)\right)^2}\log\frac{1}{a(x,y)}\lambda(dydx)+(1-V_n)\int_{B_{\varrho^n}(Y_i)}\log\frac{1-V_n}{r(x)}\lambda(dx)\right)\nonumber\end{align*}
Recalling that we only consider $n > n_0$, so that $V_n < 1$, we obtain the upper bound 
\begin{align}
& \int _{\{x: N^n (x) \geq 1\}} R(q^n(x,\cdot)\parallel K(x,\cdot))\nu_s^n(dx) \nonumber \\
& \quad \le\frac{1}{n}\frac{1}{V_n}\sum_{i=1}^n\left(\iint_{\left(B_{\varrho^n}(Y_i)\right)^2}\log\frac{1}{a(x,y)}\lambda(dydx)+\int_{B_{\varrho^n}(Y_i)}\log\frac{1}{r(x)}\lambda(dx)\right).\label{ineq:rateFunc_nu_s_n}
\end{align}
From the definition of $\varrho^n$ (see \eqref{def:varrhon}), it holds that $\varrho^n\le a(Y_i,Y_i)$ and $\varrho^n\le\Delta_{\frac{1}{n}}(Y_i)$ for all $i=1,\dots,n$. 
Moreover, the definition of $\Delta_{\frac{1}{n}}$ implies that, for a fixed $i=1,\dots,n$ and $(x,y)\in\left(B_{\varrho^n}(Y_i)\right)^2$,
\begin{align}
    \log\frac{1}{a(x,y)}&= -\log a(x,y) + \log a(Y_i, Y_i) - \log a(Y_i ,Y_i) \nonumber \\
    &< -\log a(Y_i, Y_i ) + \frac{1}{n}\nonumber\\
    &\le -\log \varrho^n + \frac{1}{n}\nonumber\\
    &= -\log \left(C_d V_n^{\frac{1}{d}}\right) + \frac{1}{n} \label{ineq:loga},
\end{align}
for some constant $C_d$ that depends on the dimension $d$ of the space $S\subseteq \mathbb{R}^d$. Similarly, for a fixed $i=1,\dots,n$ and $x\in B_{\varrho^n}(Y_i)$,
\begin{equation}
\label{ineq:logr}
    \log\frac{1}{r(x)}=-\log r(x)\le -\log r(Y_i) + \frac{1}{n}.
\end{equation}
Using the inequalities \eqref{ineq:loga} and \eqref{ineq:logr} in \eqref{ineq:rateFunc_nu_s_n} gives the upper bound
\begin{align*}
    I(\nu_s^n)&\le\frac{1}{n}\frac{1}{V_n}\sum_{i=1}^n\left(V_n^2\left(-\log \left(C_d V_n^{\frac{1}{d}}\right)+\frac{1}{n}\right)+V_n\left(-\log r(Y_i) + \frac{1}{n}\right)\right)\\
    &=-V_n\log\left(C_d V_n^{\frac{1}{d}}\right)+\frac{V_n}{n}+\frac{1}{n}\sum_{i=1}^n\log\frac{1}{r(Y_i)}+\frac{1}{n},
\end{align*}
whenever $n > n_0$. Since $V_n\to0$ by construction, we conclude that
\begin{equation*}
    \lim_{n\to\infty}I(\nu_s^n)\le\lim_{n\to\infty}\frac{1}{n}\sum_{i=1}^n\log\frac{1}{r(Y_i)}=\int_S\log\frac{1}{r(x)}\nu_s(dx)=I(\nu_s)\quad \text{a.s.},
\end{equation*}
where the second-to-last equality follows from the strong law of large numbers, and the last equality is motivated by Lemma \ref{lem:deltaIsKernelForSingularMeasure}.

\end{proof}

\begin{lemma}
    \label{lem:existenceOfCloseMeasure}
    Let $\nu \in \calP (S)$ be such that $I(\nu) < \infty$. Take $\varepsilon>0$ and $\delta>0$. There exists a probability measure $\nu^\dag\in\mathcal{P}(S)$ absolutely continuous with respect to the Lebesgue measure and such that
    \begin{equation*}
        d_{LP}(\nu^\dag,\nu)<\frac{\delta}{2} \quad \text{and}\quad I(\nu^\dag)<I(\nu)+\varepsilon.
    \end{equation*}
\end{lemma}
\begin{proof}
First, if $\nu \ll \lambda$ there is nothing to prove. Therefore, suppose this does not hold and the decomposition \eqref{eq:decomposition} is non-trivial.

    Sample $\{Y_i\}_{i=1}^\infty$ i.i.d.\ $\nu_s$ and define the sequence of random probability measures $\{\nu_s^n\}_{n\in\mathbb{N}}$ as in the construction in Lemma~\ref{lem:construction}. Motivated by the decomposition \eqref{eq:decomposition} for $\nu$, we define a new sequence of random probability measures $\{\nu^n\}_{n\in\mathbb{N}}$ by
\begin{equation*}
    \nu^n = (1-p)\cdot\nu_\lambda + p\cdot \nu_s^n,
\end{equation*}
where $p \in [0,1]$ is the same as in \eqref{eq:decomposition}, again suppressing in the notation that $p$ depends on $\nu$. By part \ref{item:abscont} of Lemma \ref{lem:construction}, $\nu ^n _s \ll \lambda$ for all $n$. It follows that $\nu^n\ll\lambda$. Moreover, from part \ref{item:weakconv} of the same Lemma, $\nu ^n$ converges weakly to $\nu$ $\nu_s$-a.s. Therefore, for any $\omega \in \Omega$ outside of a $\nu_s$-null set, there is an $N_{\delta} = N_\delta (\omega) \in \bN$ such that 
\[
    d_{LP} (\nu ^n(\omega), \nu) < \frac{\delta}{2}, \ \ \forall n \geq N_\delta (\omega).
\]

Consider now $I(\nu ^n)$. By convexity,
\begin{equation*}
    I(\nu^n)\le (1-p)\cdot I(\nu_\lambda)+p\cdot I(\nu_s^n),
\end{equation*}
for which the right-hand-side is finite w.p.\ 1 whenever $n \geq n_0$. Combined with part \ref{item:convRate} of Lemma \ref{lem:construction}, this yields that, $\nu_s$-a.s.,
\begin{equation*}
    \lim_{n\to\infty}I(\nu^n) \leq (1-p)I(\nu_\lambda)+p\cdot I(\nu_s)=I(\nu). 
\end{equation*}
Similar to before, this implies that for any $\omega \in \Omega$ outside of a $\nu_s$-null set, there is a $N_\varepsilon = N_\varepsilon (\omega) \in \bN$, such that
\[
    I (\nu ^n (\omega)) < I (\nu) + \varepsilon, \ \ \forall n \geq N_\varepsilon (\omega).
\]
As a consequence, for any $\omega \in \Omega$ outside of a null set, we can define 
\[
N(\omega) = \max\{N_\delta (\omega),N_\varepsilon (\omega), n_0\}.
\]
Then, for $n \geq N(\omega)$, $\nu ^n(\omega) \ll \lambda$,
$d_{LP}(\nu ^n (\omega), \nu) < \delta/2$, and $I(\nu ^n (\omega)) < I(\nu) + \varepsilon$. Since this is outside a $\nu_s$-null set, it has positive probability also under $\nu$. This proves the existence of a measure $\nu ^\dag$ with the claimed properties. 
\end{proof}
We emphasise that the randomness of the sequence $\{\nu^n\}$ is entirely due to the sequence of random variables $\{Y_i\}_{i=1}^\infty$. Thus, the set of outcomes of $\{Y_i\}$ that lead to a measure $\nu^n$ with the desired properties is a set with strictly positive probability. 
This guarantees the existence of a measure $\nu ^\dag$ with the claimed properties. The following result is a version of Lemma 6.17 in \cite{BudhirajaAmarjit2019AaAo}. 
\begin{lemma}
    \label{lem:existenceOfNu*}
   Let $\nu\in\mathcal{P}(S)$ satisfy $I(\nu)<\infty$.  Under \ref{ass:targetAbsContLambda}-\ref{ass:proposalDistributionAbsCont}, for given $\varepsilon>0$ and $\delta>0$, there exists $\nu^*\in\mathcal{P}(S)$ with the following properties:
    \begin{enumerate}[label=(\alph*)]
        \item \label{existence:LP}$d_{LP}(\nu^*,\nu)<\delta$; 
        \item \label{existence:AC} $\nu^*\ll\pi$ and $\pi\ll\nu^*$;
        \item \label{existence:Rate}there exists a transition probability function $q^*(x,dy)$ on $S$ such that $\nu^*$ is an invariant measure of $q^*(x,dy)$, the
associated Markov chain is ergodic, and 
\begin{equation}
\label{ineq:boundRateFuncnu*}
    I(\nu^*)\le I(\nu)+\varepsilon.
\end{equation}
    \end{enumerate}
\end{lemma}

\begin{proof}
   To prove \ref{existence:LP}, by Lemma~\ref{lem:existenceOfCloseMeasure}, there exists a measure $\nu^\dag$ that satisfies
    \begin{equation}
    \label{ineq:nuNandRateFuncCloseToLim}
        d_{LP}(\nu^\dag,\nu)<\frac{\delta}{2}\quad\text{and}\quad I(\nu^\dag)<I(\nu)+\varepsilon
    \end{equation}
    Define $\nu^*$ by
    \begin{equation}
        \label{def:nu*}
        \nu^*=\left(1-\frac{\delta}{4}\right)\nu^\dag+\frac{\delta}{4}\pi.
    \end{equation}
 Then,
    \begin{equation*}
    d_{LP}(\nu^*,\nu^\dag)\le\|\nu^*-\nu^\dag\|_{TV}=\Bigg\lVert\left(1-\frac{\delta}{4}\right)\nu^\dag+\frac{\delta}{4}\pi-\nu^\dag\Bigg\rVert_{TV}=\frac{\delta}{4}\|\pi-\nu^\dag\|_{TV}\le \frac{\delta}{2}.
\end{equation*}
Combining this with \eqref{ineq:nuNandRateFuncCloseToLim} and the triangle inequality now yields the desired upper bound on $d_{LP} (\nu ^*, \nu)$,
\begin{equation*}
    d_{LP}(\nu^*,\nu)\le d_{LP}(\nu^\dag,\nu)+d_{LP}(\nu^*,\nu^\dag)<\delta.
\end{equation*}

\ref{existence:AC}. The first part of  follows from $\nu ^\dag \ll \lambda$ (see Lemma \ref{lem:existenceOfCloseMeasure}) and the fact that, by Assumption~\ref{ass:targetAbsContLambda}, $\lambda \ll \pi$. For the second part, for any $A\in\mathcal{B}(S)$, $\nu^*(A)\ge\frac{\delta}{4}\pi(A)$ by construction. Thus, $\pi\ll\nu^*$. 

We now prove part \ref{existence:Rate}, following the steps in \cite[Lemma 6.17]{BudhirajaAmarjit2019AaAo}. Since $I(\nu ^\dag) < \infty$, by Lemma~6.8(b) in \cite{BudhirajaAmarjit2019AaAo}, we can choose a transition kernel $q(x,dy)$ with invariant measure $\nu^\dag$ and
\begin{equation*}
    \int_S R(q(x,\cdot)\parallel K(x,\cdot))\nu^\dag(dx)=I(\nu^\dag).
\end{equation*}
Define the $\gamma ^\dag$, $\theta$ and $\gamma ^*$ in $\mathcal{P}(S^2)$ by,
\begin{align*}
    &\gamma^\dag = \nu^\dag\otimes q,\\
    &\theta = \pi\otimes K,
   \end{align*}
   and
   \begin{align*}
    &\gamma^* = \left(1-\frac{\delta}{4}\right)\gamma^\dag+\frac{\delta}{4}\theta.
\end{align*}
Both marginals of $\gamma ^\dag$ equal $\nu ^\dag$. Similarly, both marginals of $\theta$ equal $\pi$. From \eqref{def:nu*} it then follows that both marginals of $\gamma ^*$ equal $\nu ^*$. From Lemma~6.8(a) in~\cite{BudhirajaAmarjit2019AaAo}, there exists a transition kernel $q^*(x,dy)$ that has $\nu^*$ as invariant probability distribution and such that $\gamma^*=\nu^*\otimes q^*$. Using the convexity of relative entropy, the property $R(\alpha\parallel\alpha)=0$ and \eqref{ineq:nuNandRateFuncCloseToLim}, we obtain the upper bound \eqref{ineq:boundRateFuncnu*}:
\begin{align*}
    I(\nu^*)&\le\int_SR(q^*(x,\cdot)\parallel K(x,\cdot))\nu^*(dx)\\
    &=R(\gamma^*\parallel\nu^*\otimes K)\\
    &=R\left(\left(1-\frac{\delta}{4}\right)\nu^\dag\otimes q+\frac{\delta}{4}\pi\otimes K\Big\Vert\left(1-\frac{\delta}{4}\right)\nu^\dag\otimes K+\frac{\delta}{4}\pi\otimes K\right)\\
    &=\left(1-\frac{\delta}{4}\right)R(\nu^\dag\otimes q \parallel \nu^\dag \otimes K) + \frac{\delta}{4} R(\pi\otimes K\parallel \pi\otimes K) \\
    &=\left(1-\frac{\delta}{4}\right) I(\nu^\dag)\\
    &< I(\nu)+\varepsilon.
\end{align*}

It remains to show that the Markov process associated with $q^*$ is ergodic. Let $f = \frac{d \nu^*}{d\pi}$ be the Radon-Nikodym derivative of $\nu^*$ with respect to $\pi$, which is well-defined by part \ref{existence:AC}. Since $\nu^*(A)\ge\frac{\delta}{4}\pi(A)$ for all $A\in\mathcal{B}(S)$, for all $x \in S$, $f(x)\ge \frac{\delta}{4}$. We observe that for any $A,B\in\mathcal{B}(S)$,
\begin{equation*}
    \gamma^*(A\times B) = \int_{A}q^*(x,B)\nu^*(dx)=\int_{A}q^*(x,B)f(x)\pi(dx),
\end{equation*}
and, from the definition of $\gamma ^*$,
\begin{equation*}
    \gamma^*(A\times B) \ge \frac{\delta}{4}\theta(A\times B) = \frac{\delta}{4}\int_A K(x,B)\pi(dx).
\end{equation*}
It follows that 
\begin{equation*}
    q^*(x,B)\ge \frac{\delta}{4f(x)}K(x,B),\quad \forall x,\,\pi-\text{a.s.},
\end{equation*}
for all $B\in\mathcal{B}(S)$. Thus, $\pi$-.a.s.\ for $x \in S$, $K(x,\cdot) \ll q^*(x,\cdot)$. To show absolute continuity in the reverse direction, note that from
\begin{equation*}
    \int_SR(q^*(x,\cdot)\parallel K(x,\cdot))\nu^*(dx)<\infty,   
\end{equation*}
it follows that $R(q^*(x,\cdot)\parallel K(x,\cdot))<\infty$. Thus, $q^*(x,\cdot)\ll K(x,\cdot)$, $\nu^*$-a.s. As $\nu^*$ and $\pi$ are equivalent measures, we obtain that $q^*(x,\cdot)$ and $K(x,\cdot)$ are equivalent $\pi$-a.s. This means that there exists a Borel set $C\in\mathcal{B}(S)$ such that $\pi(C)=0=\nu^*(C)$, and $q^*(x,\cdot)$ and $K(x,\cdot)$ are equivalent for all $x$ in the complement of $C$. If we redefine $q^*(x,\cdot)=K(x,\cdot)$ for all $x\in C$, we obtain the equivalence between $q^*(x,dy)$ and $K(x,\cdot)$ for all $x\in S$. Besides, being $\nu^*(C)=0$, the newly defined $q^*(x,\cdot)$ still has $\nu^*$ as invariant measure. To show that $q^*(x,\cdot)$ is ergodic, recall that in Remark~\ref{rmk:KhasAlawaysAcceptancePart} we proved that there are no disjoint Borel sets $A_1,A_2\in\mathcal{B}(S)$ such that
 \begin{equation*}
        K(x,A_1)=1\;\forall x\in A_1\quad\text{and}\quad K(y,A_2)=1\;\forall y\in A_2.
\end{equation*}
Because $q^*(x,\cdot)$ and $K(x,\cdot)$ are equivalent for all $x\in S$, it follows that also $q^*(x,\cdot)$ satisfies the property that there do not exist disjoint $A_1,A_2\in\mathcal{B}(S)$ for which
 \begin{equation*}
        q^*(x,A_1)=1\;\forall x\in A_1\quad\text{and}\quad q^*(y,A_2)=1\;\forall y\in A_2,
\end{equation*}
meaning that $q^*(x,\cdot)$ is indecomposable. Therefore, by Theorem 7.16 in~\cite{Breiman}, $\nu^*$ is the unique invariant distribution for $q^*(x,dy)$ and the Markov chain associated with $\nu^*$ and $q^*(x,dy)$ is ergodic.
\end{proof}

We are ready to prove Proposition \ref{lemma:lower}, the Laplace principle lower bound. The following proof is mostly based on the proof of Proposition~6.15 in \cite{BudhirajaAmarjit2019AaAo},
with minor changes due to the lack of Condition \ref{cond:DE}. The main work has been done in Lemmas \ref{lem:deltaIsKernelForSingularMeasure}-\ref{lem:existenceOfNu*}, and most of the proof from \cite{BudhirajaAmarjit2019AaAo} now goes through, with some minor modifications to rely on those results rather than Condition \ref{cond:DE}. We include the full argument for self-containment and convenience for the reader.

\begin{proof}[Proof of Proposition \ref{lemma:lower}]
To prove the Laplace lower bound \eqref{ineq:LaplaceLowerBound}, it is sufficient to consider only bounded Lipschitz continuous functions $F$ (see Corollary~1.10 in~\cite{BudhirajaAmarjit2019AaAo}). Since we have endowed $\calP (S)$ with the L\'evy-Prohorov metric $d_{LP}$, a function $F \in C_b(\calP(S))$ is Lipschitz if
    \begin{equation*}
        \sup_{\nu_1\neq\nu_2}\frac{\left|F(\nu_1)-F(\nu_2)\right|}{d_{LP}(\nu_1,\nu_2)}<\infty.
    \end{equation*}

Recall that $X = \{X_i\}_{i\geq 0}$ denotes the Metropolis-Hastings chain, as described in Section \ref{sec:MH}, and $\{ L^n \}_n$ the associated sequence of empirical measures. We now construct a nearly optimal sequence of controls in variational representation \eqref{eq:rep},
\begin{align}
    -\frac{1}{n} \log \bE \left[ e^{-nF(L^n)} \right] &= \inf _{\{ \bar \mu ^n _i \} } \bE \left[ F(\bar L ^n) + \frac{1}{n} \sum_{i=1} ^n R(\bar \mu ^n _i \parallel K(\bar X ^n _i, \cdot)) \right]. 
\end{align}

Let $\varepsilon>0$ be given and choose $\nu\in\mathcal{P}(S)$ such that 
\begin{equation}
\label{ineq:FPlusICloseToInf}
    F(\nu)+I(\nu)\le \inf_{\mu\in\mathcal{P}(S)}\left[F(\mu)+I(\mu)\right] + \varepsilon <\infty
\end{equation}
Since $F$ is continuous, there exists $\delta>0$ such that $d_{LP}(\mu,\nu)<\delta$ implies $\left|F(\mu)-F(\nu)\right|<\varepsilon$. In Lemma~\ref{lem:existenceOfNu*} it is shown that, for any such pair $\delta, \varepsilon$, there exists a probability measure $\nu^*\in\mathcal{P}(S)$ and a transition probability $q^*(x,dy)$ such that $\nu^*$ is an invariant measure for $q^*(x,dy)$, the Markov chain with initial distribution $\nu^*$ and transition probability $q^*(x,dy)$ is ergodic, besides
\begin{equation}
\label{ineq:rateFuncNu*Bound}
    I(\nu^*) \leq \int_SR(q^*(x,\cdot)\parallel K(x,\cdot))\nu^*(dx) \le I(\nu) +\varepsilon<\infty. 
\end{equation}
Moreover, Part \ref{existence:LP} of the Lemma ensures $d_{LP}(\nu^*,\nu)<\delta$, which then implies
\begin{equation}
\label{ineq:FPlusEps}
    F(\nu^*)\le F(\nu)+\varepsilon.
\end{equation}
Thus, $\nu ^*$ is such that
\[
    F(\nu ^*) + I(\nu ^*) \leq F(\nu) + I(\nu) + 2\varepsilon.
\]
The transition probability function $q^*$ associated with $\nu ^*$ is now used to define the controls, 
\begin{equation}
\label{eq:control}
    \bar\mu_i^n(dy)=q^*(\bar X_{i-1}^n,dy), \ \ i =1, \dots, n.
\end{equation}
With the inequalities \eqref{ineq:rateFuncNu*Bound}-\eqref{ineq:FPlusEps} established, and the choice \eqref{eq:control} for the controls, we can proceed with the same arguments as in the proof of Proposition 6.15 in \cite{BudhirajaAmarjit2019AaAo}. 

With the choice \eqref{eq:control}, the running costs for the controlled chain $\bar X ^n$ are
 \begin{equation*}
     \frac{1}{n}\sum_{i=0}^{n-1}R(\bar\mu_i^n(\cdot)\parallel K(\bar X_i^n,\cdot))=\frac{1}{n}\sum_{i=0}^{n-1}R(q^*(\bar X_i^n,\cdot)\parallel K(\bar X_i^n,\cdot)).
 \end{equation*}
The $\bar \mu ^n _i$s only give the conditional distributions for $\bar X^n _i$ for $i=1, \dots, n$. 
For the distribution of the initial point $\bar X^n _0$, consider two choices: $\delta _x$ and $\nu ^*$. Let $\bP_x$ and $\bP^*$ denote the corresponding probability measures and let $\mathbb{E}_x$ and $\mathbb{E}^*$ be the associated expectation, respectively. Define $D^n$ and $D^n _x$ as the expected difference between the empirical average of the relative entropy between $q^*$ and $K$, and its mean, under $\bP^*$ and $\bP_x$, respectively,
\begin{equation*}
D^n = \mathbb{E}^*\left[\left|\frac{1}{n}\sum_{i=0}^{n-1}R(q^*(\bar X_i^n,\cdot)\parallel K(\bar X_i^n,\cdot))-\int_SR(q^*(\xi,\cdot)\parallel K(\xi,\cdot))\nu^*(d\xi)\right|\right],
\end{equation*}
and 
\begin{equation*}
D^n_x = \mathbb{E}_x\left[\left|\frac{1}{n}\sum_{i=0}^{n-1}R(q^*(\bar X_i^n,\cdot)\parallel K(\bar X_i^n,\cdot))-\int_SR(q^*(\xi,\cdot)\parallel K(\xi,\cdot))\nu^*(d\xi)\right|\right].
\end{equation*}

From the definition of the controls \eqref{eq:control}, and since $\nu^*$ is an invariant measure of $q^*(x,dy)$, all terms of the controlled process $\{\bar X_i^n \}_{i=0}^n$ are distributed according to $\nu^*$. By the non-negativity of the relative entropy and $R(\cdot\parallel\cdot)$ and \eqref{ineq:rateFuncNu*Bound}, we obtain
\begin{align*}
    \mathbb{E}^*\left[\left|R(q^*(\bar X_i^n,\cdot)\parallel K(\bar X_i^n,\cdot))\right|\right]&
    =\int_SR(q^*(\xi,\cdot)\parallel K(\xi,\cdot))\nu^*(d\xi) 
    \le I(\nu)+\varepsilon<\infty.
\end{align*}
The $L^1$-ergodic theorem~\cite[Corollary 6.25]{Breiman} 
then gives
\begin{equation*}
    \lim_{ n\to\infty}D^n=0.
\end{equation*}
Moreover, note that $D^n=\int_SD^n_x\nu^*(dx)$. Therefore, the convergence of $D^n$ also implies that 
\begin{equation*}
    \lim_{n\to\infty}\int_SD^n_x\nu^*(dx)=0.
\end{equation*}
Convergence in probability of $D^n _x$ to 0 now follows from Chebyshev's inequality: for any $c>0$,
\begin{equation*}
    \lim_{n\to\infty}\nu^*\{x\in S \,:\,D_x^n\ge c\} \leq \lim _{n \to \infty} \frac{1}{c} \int _{\{x:  D^n _x \geq c \}} D^n _x \nu ^*(dx) \leq \frac{1}{c} \lim _{n \to \infty} D^n =0.
\end{equation*}
From this convergence in probability, for every subsequence of $\{ n\}$ there is a further subsequence, which we also denote by $\{ n \}$, such that the convergence is w.p.\ 1. That is, there is a Borel set $\Phi _1$ with $\nu ^*(\Phi_1) =1$, such that along such (sub)subsequences and for all $x \in \Phi _1$, 
\begin{equation}
\label{eq:limRunningCost}
    \lim_{n\to\infty}\mathbb{E}_x\left|\frac{1}{n}\sum_{i=0}^{n-1}R(q^*(\bar X_i^n,\cdot)\parallel K(\bar X_i^n,\cdot))-\int_SR(q^*(\xi,\cdot)\parallel K(\xi,\cdot))\nu^*(d\xi)\right|=0.
\end{equation}
Abusing notation, we now fix such a subsubsequence $\{ n \}$. The previous argument show the a.s.\ convergence of the running costs and we now consider the corresponding sequence of controlled empirical measures $\{\bar L ^n\}$. Because $S \subset \bR ^d$, there is a countable convergence-determining class $\Xi \subset C_b (S)$ (see e.g.\ Appendix A in \cite{BudhirajaAmarjit2019AaAo}). For each $g\in\Xi$, we define the set
\begin{equation*}
    \calA (g) = \left\{\omega\in\Omega\,:\,\lim_{n\to\infty}\frac{1}{n}\sum_{i=0}^{n-1}g(\bar X_i^n(\omega))=\int_Sg(x) \nu^* (dx)\right\}.
\end{equation*}
By the pointwise ergodic theorem~\cite[Sect. 6.5]{Breiman}, 
\begin{equation*}
    \bP^*\left\{\calA(g)\right\}=1.
\end{equation*}
Observing that $\bP^*\left\{\calA(g)\right\}=\int_S \bP_x\left\{\calA(g)\right\}\nu^*(dx)$, we obtain
\begin{equation*}
    \int_S \bP_x\left\{\calA(g)\right\}\nu^*(dx)=1.
\end{equation*}
This implies that $\bP_x\left\{\calA(g)\right\}=1$ a.s., i.e., there exists a Borel set $\Phi_2(g)\in\mathcal{B}(S)$ with $\nu^*(\Phi_2(g))=1$ and such that $\bP_x\left\{\calA(g)\right\}=1$ for $x\in\Phi_2(g)$. 

To establish the convergence of $\bar L ^n$, we define $\Phi_2 = \cap_{g\in\Xi}\Phi_2(g)$. Since $\Xi$ is countable, $\Phi _2$ satisfies $\nu^*(\Phi_2)=1$. Then, for all initial points $\bar X^n _0 = x\in\Phi_2$,
\begin{equation*}
    \lim_{n\to\infty}\int_S g\,d\bar L^n=\lim_{n\to\infty}\frac{1}{n}\sum_{i=0}^{n-1}g(\bar X_i^n)=\int_Sg\,d\nu^*,
\end{equation*}
$\bP_x-$a.s.\ for all $g\in\Xi$. Because $\Xi$ a convergence determining class, it follows that $\bar L^n\Rightarrow\nu^*$ $\bP_x-$a.s. for all $x\in\Phi_2$. From the continuity of $F$ on $\mathcal{P}(S)$, we then have
\begin{equation}
\label{eq:limFL}
    \lim_{n\to\infty}F(\bar L^n)=F(\nu^*),
\end{equation}
for all $x\in\Phi_2$. 

We now combine the arguments for the running costs and the controlled empirical measures to show the Laplace principle lower bound on a set of $\nu ^*$-measure 1. Define the set $\Phi=\Phi_1\cap\Phi_2 \subset S$. Since $\nu^*(\Phi)=\nu^*(\Phi_2)=1$, we have $\nu ^*(\Phi ) =1$. For all $x \in \Phi$, both \eqref{eq:limRunningCost} and \eqref{eq:limFL} hold, and
\begin{align*}
    \limsup_{n\to\infty}-\frac{1}{n}\log\mathbb{E}e^{-nF(L^n)}&\le\lim_{n\to\infty}\mathbb{E}_x\left[F(\bar L^n)+\frac{1}{n}\sum_{i=0}^{n-1}R(q^*(\bar X_i^n,\cdot)\parallel K(\bar X_i^n,\cdot))\right]\\
    &=F(\nu^*)+\int_SR(q^*(\xi,\cdot)\parallel K(\xi,\cdot))\nu^*(d\xi)\\
    &\le F(\nu)+I(\nu)+2\varepsilon\\
    &\le \inf_{\mu\in\mathcal{P}(S)}\left[F(\mu)+I(\mu)\right]+3\varepsilon,
\end{align*}
where the inequality on the third line comes from \eqref{ineq:rateFuncNu*Bound} and \eqref{ineq:FPlusEps}, while the inequality on the last line follows from \eqref{ineq:FPlusICloseToInf}. By taking the limit $\varepsilon\to0$ we obtain the upper bound \eqref{ineq:LaplaceLowerBound} for all $x\in\Phi$.

We conclude the proof by extending this result from $\Phi$ to the whole space $S$. Whereas \cite{BudhirajaAmarjit2019AaAo} relies on the transitivity condition \eqref{eq:transCond} for this extension, we instead rely on the properties of the MH kernel; this requires only minor changes in the argument.

By Lemma~\ref{lem:existenceOfNu*},  $\nu^*$ and $\pi$ are equivalent, thus $\nu^*(\Phi)=1$ implies $\pi (\Phi) =1$.  Moreover, $\pi$ and $\lambda$ are equivalent measures by Assumption \ref{ass:targetAbsContLambda}, and we have
\begin{equation*}
    a(x,\Phi)=\int_\Phi a(x,y)dy=\int_S a(x,y)dy=a(x,S),
\end{equation*}
for all $x\in S$. As a consequence, $K(x,\Phi)\ge a(x,\Phi)=a(x,S)$, which is strictly positive for all $x\in S$ (see Remark~\ref{rmk:KhasAlawaysAcceptancePart}). It follows that
\begin{equation}
\label{ineq:KinPhipositive}
    K(x,\Phi)>0,\qquad \forall x\in S.
\end{equation}

Define $\tilde L ^n$ as the empirical measure of $X_1, \dots , X_n$,
\begin{equation*}
    \tilde L^n=\frac{1}{n}\sum_{i=1}^n\delta_{X_i}.
\end{equation*}
Since $L^n $ and $\tilde L ^n$ only differ in the first and last summands,
\begin{equation*}
    \lVert\tilde L^n -L^n\rVert_{TV}\le\frac{2}{n}.
\end{equation*}
Let $L_F<\infty$ denote the Lipschitz constant of $F$ with respect to the L\'{e}vy-Prohorov metric. For all $\omega\in\Omega$,
\begin{equation*}
    F(L^n)\le F(\tilde L^n)+ L_F \cdot d_{LP}(L^n,\tilde L^n)\le F(\tilde L ^n) + L_F \norm{L^n - \tilde L^N} _{TV} \le F(\tilde L^n)+\frac{2L_F}{n}.
\end{equation*}
Take any $x\in S$ and $n\in\mathbb{N}$. Since the $X^n _i$s evolve according to $K$, using the previous inequality we have,
\begin{align}
    \mathbb{E}_x\left[e^{-nF(L^n)}\right]&\ge e^{-2L_F} \mathbb{E}_x\left[e^{-nF(\tilde L^n)}\right]\nonumber\\
    &=e^{-2L_F}\int_S\mathbb{E}\left[e^{-nF(\tilde L^n)}\mid X_1=y\right]\,K(x,dy)\nonumber\\
    &=e^{-2L_F}\int_S\mathbb{E}_y\left[e^{-nF( L^n)}\right]\,K(x,dy)\nonumber\\
    &\ge  e^{-2L_F}\int_\Phi\mathbb{E}_y\left[e^{-nF(\tilde L^n)}\right]\,K(x,dy),\label{ineq:extendLowerBound}
\end{align}
where the equality on the third line is due to the Markov property. With this lower bound established, from here we can again follow the proof of Proposition 6.15 in \cite{BudhirajaAmarjit2019AaAo}. Let $\varepsilon>0$ be fixed. We have that \eqref{ineq:LaplaceLowerBound} holds for all $y\in\Phi$, why for each such $y$ there exists an $N(y,\varepsilon)\in\mathbb{N}$ such that 
\begin{equation}
\label{ineq:lowerBoundForY}
    -\frac{1}{n}\log\mathbb{E}_y\left[e^{-nF(L^n)}\right]\le \inf_{\mu\in\mathcal{P}(S)}\left[F(\mu)+I(\mu)\right]+\varepsilon
\end{equation}
for all $n\ge N(y,\varepsilon)$. Without loss of generality, take $N(y,\varepsilon)$ as the smallest integer with this property. Then, the function $S\to\mathbb{N}$ that maps $y$ into $N(y,\varepsilon)$ is measurable, the sets 
\begin{equation*}
    \Phi^{(i)} = \{y\in\Phi \,:\, N(y,\varepsilon)=i\}\subset S
\end{equation*}
are disjoint Borel sets, and $\Phi=\cup_{i=1}^\infty\Phi^{(i)}$.

Because $K(x, \Phi) > 0$ for all $x \in S$ (see \eqref{ineq:KinPhipositive}), we have that for all $x\in S$ there exists an $i_0\in\mathbb{N}$ such that $K(x,\Phi^{(i_0)})>0$. Combined with the bounds in \eqref{ineq:extendLowerBound}, and  \eqref{ineq:lowerBoundForY}, this implies that for all $n\ge i_0$,
\begin{align*}
    \mathbb{E}_x\left[e^{-nF(L^n)}\right]&\ge e^{-2L_F}\int_\Phi\mathbb{E}_y\left[e^{-nF(\tilde L^n)}\right]\,K(x,dy)\\
    &\ge e^{-2L_F}\int_{\Phi^{(i_0)}}\mathbb{E}_y\left[e^{-nF(\tilde L^n)}\right]\,K(x,dy)\\
    &\ge e^{-2L_F} \exp\left\{-n\left(\inf_{\mu\in\mathcal{P}(S)}[F(\mu)+I(\mu)]+\varepsilon\right)\right\} K(x,\Phi^{(i_0)}).
\end{align*}
It follows that
\begin{align*}
    \limsup_{n\to\infty}-\frac{1}{n}\log\mathbb{E}_x \left[e^{-nF(L^n)} \right] &\le \inf_{\mu\in\mathcal{P}(S)}[F(\mu)+I(\mu)]+\varepsilon+\lim_{n\to\infty}\frac{2L_F-\log K(x,\Phi^{(i_0)})}{n} \\
    &= \inf_{\mu\in\mathcal{P}(S)}[F(\mu)+I(\mu)]+\varepsilon.
\end{align*}
In the limit $\varepsilon\to0$, we have for all $x\in S$,
\begin{equation*}
    \limsup_{n\to\infty}-\frac{1}{n}\log\mathbb{E}_x\left[ e^{-nF(L^n)} \right]\le\inf_{\mu\in\mathcal{P}(S)}[F(\mu)+I(\mu)]+\varepsilon.
\end{equation*}
This concludes the proof of the Laplace principle lower bound.
\end{proof}

\subsection*{Acknowledgments}
We thank Professors I.~ Kontoyiannis and S.~P.~Meyn for comments on the first version of the paper, and for pointing out their previous work \cite{KM03, KM05}, and Prof.\ A.~Wang for insightful comments that lead to a refinement of Assumption \ref{ass:proposalDistributionAbsCont}.

The research of FN and PN was supported by the Swedish e-Science Research Centre (SeRC). PN was also supported by Wallenberg AI, Autonomous Systems and Software Program (WASP) funded by the Knut and Alice Wallenberg Foundation, and by the Swedish Research Council (VR-2018-07050).

\bibliographystyle{alpha}
\bibliography{references}

\end{document}